\documentclass[11pt]{amsart}
\usepackage{amsthm}
\usepackage{siunitx}
\usepackage{mathpazo}
\usepackage{mathtools}
\usepackage[all]{xy}
\usepackage[headings]{fullpage}
\usepackage{paralist}
\usepackage{enumitem}
\usepackage{graphicx, amsmath,latexsym, amsfonts, amssymb, amstext}
\usepackage{mathrsfs}
\usepackage{color}
\usepackage[utf8]{inputenc}
\usepackage[english]{babel}
\usepackage{hyperref}
\hypersetup{
    colorlinks=true,
    linkcolor=blue,
   filecolor=blue, citecolor=blue     
}

\usepackage{textcomp}
\usepackage[small,bf]{caption}
\usepackage{amssymb, amsmath, amsfonts, amsthm}
\usepackage{epsfig,amsfonts,amsbsy,amssymb,amsthm}
\usepackage[numbers]{natbib}
\allowdisplaybreaks

\newtheorem{theorem}{Theorem}
\newtheorem{lemma}[theorem]{Lemma}

\newtheorem{proposition}[theorem]{Proposition}

\theoremstyle{definition}

\newtheorem{definition}[theorem]{Definition}
\newtheorem{example}[theorem]{Example}
\newtheorem{remark}[theorem]{Remark}
\numberwithin{equation}{section}

\newtheorem{discussion}[theorem]{Discussion}
\newenvironment{discussionbox}[1][]{%
    \begin{discussion}[#1]\pushQED{\qed}}{\popQED \end{discussion}}

\numberwithin{theorem}{section}
\DeclareMathOperator{\syz}{Syz}

\newcommand{\ints}{\mathbb{Z}}

\newcommand{\p}{\mathfrak p}
\newcommand{\m}{\mathfrak m}
\DeclareMathOperator{\h}{H}

\def\to{\longrightarrow}
\def\ov{\overline}

\DeclareMathOperator{\assh}{Assh}
\DeclareMathOperator{\ass}{Ass}

\DeclareMathOperator{\ann}{Ann}

\DeclareMathOperator{\depth}{depth}

\DeclareMathOperator{\reg}{reg}

\begin{document}

\title[Finiteness of Hilbert Coefficients]
{On the Finiteness of the set of Hilbert Coefficients}

\date{\today}

\thanks{The first author is supported by INdAM COFOUND Fellowships cofounded by Marie Curie actions, Italy.}


% \titlespacing*{\section}{0pt}{2\baselineskip}{1\baselineskip}
% \title[ On the finiteness of the set of the first Hilbert coefficients]{On the finiteness of the set of the first Hilbert coefficients}
\author[Masuti]{Shreedevi K. Masuti}
\address{Dipartimento Di Matematica, Universit\`a Di Genova, Via Dodecaneso 35, 16146, Genova,
Italy}
\email{masuti@dima.unige.it,masuti.shree@gmail.com}
\author[Saloni]{Kumari Saloni}
\address{Department of Mathematics, IIT Guwahati, Guwahati, Assam 781039, India}
\email{saloni.kumari@iitg.ac.in}

\begin{abstract}
\begin{sloppypar}
Let $(R,\m)$ be a Noetherian local ring of dimension $d$ and $K,Q$ be $\m$-primary ideals in $R.$ In this paper we study the finiteness properties of the sets $\Lambda_i^K(R):=\{g_i^K(Q): Q\mbox{ is a parameter 
ideal of }R\},$ where $g_i^K(Q)$ denotes the Hilbert coefficients of $Q$ with respect to $K,$ for $1\leq i \leq d.$ We prove that $\Lambda_i^K(R)$ is finite for all $1\leq i \leq d$ if and only if $R$ is generalized Cohen-Macaulay. Moreover, we show that if $R$ is unmixed then finiteness of the set $\Lambda_1^K(R)$ suffices to conclude that $R$ is generalized Cohen-Macaulay. We obtain partial results for $R$ to be Buchsbaum in terms of $|\Lambda_i^K(R)|=1.$ Our results are more general than in \cite{vanishing2} and \cite{goto-ozeki}. We also obtain a criterion for the set $\Delta^K(R):=\{g_1^K(I): I\mbox{ is an  $\m$-primary ideal of }R\}$ to be finite, generalizing a result of \cite{KT}.
\end{sloppypar}
\end{abstract}

\maketitle
\section{Introduction} 
The main objective of this paper is to study the finiteness properties of various sets of the Hilbert coefficients relative to the properties of the ring.
% We give a characterization for ring to be generalized Cohen-Macaulay or 
% Buchsbaum in terms of these sets. 
First we introduce the notations needed to define these sets.

Throughout this paper $(R,\m)$ denotes a Noetherian local ring of dimension $d$ with maximal ideal
$\m ,$ $M$ a finitely generated $R$-module of dimension $r$ and $K$ a fixed $\m$-primary ideal.
% Let $(R,\m)$ be a Noetherian local ring of dimension $d$ and $K, Q$ be $\m$-primary ideals in $R.$   
% Let $G(Q)=\mathop{\oplus}\limits_{n\geq 0} Q^n/Q^{n+1}$ be the associated graded ring of $Q$. 

For an $\m$-primary ideal $Q,$ the fiber cone of $Q$ with respect to $K$ is the standard graded algebra $F_K(Q)=\mathop{\oplus}\limits_{n\geq 0} Q^n/KQ^n.$ The {\it Hilbert function} of the fiber cone $F_K(Q)$ is given by $H(F,n):=\ell_R(Q^n/KQ^n),$ where $\ell_R(M)$ denotes the 
length of an $R$-module $M.$ It is well known that $H(F,n)$ agrees with a polynomial $P(F,n)$ of degree $d-1,$ for $n\gg 0,$ called the {\it{Hilbert polynomial}} of $F_K(Q).$ 
We can write  $P(F,n)$ in the following way:
\begin{equation*}
 P(F,n)=\sum\limits_{i=0}^{d-1}(-1)^i f_i^K(Q)\binom{n+d-i-1}{d-1-i}
\end{equation*}
 where the coefficients $f_i^K(Q)$ are integers known as the {\it fiber coefficients} of $Q$ with respect to $K.$ 
%  Let $M$ be a finitely generated $R$-module of dimension $r.$ 
 The  
{\it{Hilbert-Samuel function}} of $Q$ for $M$ is the function $H(Q,n,M)=\ell_R(M/Q^nM).$ In \cite{Jayanthan&Verma1} 
authors introduced the Hilbert function of $Q$ with respect to $K$ defined as $ H_K(Q,n)=\ell_R(R/KQ^n) .$
It is known that for $n\gg 0$, $H(Q,n, M)$ (resp. $H_K(Q,n)$) agrees with a polynomial
$P(Q,n,M)$ (resp. $P_K(Q,n)$) of degree $r$ (resp. $d$). We can write these polynomials in the following manner:
\begin{eqnarray}
P(Q,n,M)&=& \sum\limits_{i=0}^r(-1)^ie_i(Q,M)\binom{n+r-i-1}{r-i}\label{Hilbert-S-poly}\\
P_K(Q,n)&=&\sum\limits_{i=0}^d(-1)^ig_i^K(Q)\binom{n+d-i-1}{d-i}\label{Hilbert-S-poly-for-K}
\end{eqnarray}
for unique integers $e_i(Q,M)$ (resp. $g_i^K(Q)$) known as the {\it{Hilbert coefficients}} of $Q$ for $M$ (resp. {\it{Hilbert coefficients}} of $Q$ with respect to $K$). One of the motivations to study $g_i^K(Q)$  is that these coefficients are related to the fiber coefficients (see \eqref{eqn-fi-gi-ei}) and hence are useful to study the properties of $f_i^K(Q).$ The properties of $g_i^K(Q)$ have been studied in \cite{Clare}, \cite{Gu-Zhu-Tang}, \cite{Jayanthan&Verma1}, \cite{Jayanthan&Verma}, \cite{saloni-1}, \cite{Gu-Zhu-Tang2}.   

In this paper we consider the sets
\begin{eqnarray*}
 \Lambda_i^K(R)&=& \{g_i^K(Q)~|~Q \text{ is a parameter ideal of } R\} \mbox{ and } \\
\delta_i^K(R)&=&\{g_i^K(Q)~|~Q \text{ is a parameter ideal of }R \text{ such that } Q\subseteq K\}
\end{eqnarray*}
for $1 \leq i \leq d.$
Note that $\delta_i^K(R)\subseteq \Lambda_i^K(R).$ Following the notation of \cite{vanishing2}, we set  
\begin{eqnarray*}
\Lambda_i(M)&=&\{e_i(Q,M)~|~Q \text{ is a parameter ideal for }M\} \mbox{ for all }1 \leq i \leq r.
\end{eqnarray*}
For a set $S$ we use $|S|$ to denote the cardinality of the set $S.$ 

% % \begin{eqnarray*}
% % \Lambda_i(M)&=&\{e_i(Q,M)~|~Q \text{ is a parameter ideal for }M\},\\
% % \Lambda_i^K(R)&=& \{g_i^K(Q)~|~Q \text{ is a parameter ideal of } R\}, \\
% % \delta_i^K(R)&=&\{g_i^K(Q)~|~Q \text{ is a parameter ideal of }R \text{ such that } Q\subseteq K\},\\
% % \Delta_R(M)&=&\{e_1(I,M)~|~I \text{ is an $\m$-primary ideal of }R\},\\
% % \Delta^K(R)&=&\{g_1^K(I)~|~I \text{ is an $\m$-primary ideal of }R\}.
% % \end{eqnarray*}
% In a generalized Cohen-Macaulay (resp. Buchsbaum) module $M$, the set $\Lambda_1(M)$ is known to be finite (resp. singleton), see \cite{trung-gcm}.
% % In \cite{vanishing2} authors  proved that for an unmixed module $M$, the converses also hold true. In other words, the finiteness (resp. cardinality one)
Let $r,d \geq 2.$ In \cite{vanishing2} authors  proved that the set $\Lambda_1(M)$ is finite (resp. singleton) if and only if $M$ is generalized Cohen-Macaulay (resp. Buchsbaum) provided $M$ is an unmixed module, see \cite[Theorems 4.5 and 5.4]{vanishing2}. In Section \ref{SetOfg1}, we investigate the set $\Lambda_1^K(R)$ for analogous properties. 
% We are able to give a 
% characterization for $R$ to be generalized Cohen-Macaulay in terms of  sets $\Lambda_1^K(R)$ and $\delta_1^K(R)$ provided $R$ is unmixed (Theorem \ref{GCM}). 
 We prove that an unmixed local ring $R$ is generalized Cohen-Macaulay if and only $\Lambda_1^K(R)$ ( equivalently $\delta_1^K(R)$) is finite (Theorem \ref{GCM}).
Next, we prove that if $R$ is unmixed and $|\Lambda_1^K(R)|=1$ then $R$ is Buchsbaum where as the converse holds true for $K=\m$ (Theorem \ref{thm: Buchsbaum}). We expect that $\Lambda_1^K(R)$ need not be singleton in a Buchsbaum local ring for an arbitrary $\m$-primary ideal $K$ (see Discussion \ref{disc:Buchs}).
% we prove that if $R$ is Buchsbaum, $g_1^{\m}(Q)$ is independent of the choice of parameter ideals $Q$ i.e. $|\Lambda_1^{\m}(R)|=1$.
% % We also give an explanation why this may not be true in general for an arbitrary $\m$-primary ideal $K$ and a parameter ideal $Q$.
% % Further, we show that if $R$ is unmixed and $|\Lambda_1^K(R)|=1$ then $R$ is Buchsbaum (Theorem \ref{thm: Buchsbaum}). 
% We also provide an explanation for why the converse may not be true in general for an arbitrary $\m$-primary ideal $K$ and a parameter ideal $Q$.

In Section \ref{SetOfgi}, we study the finiteness of the sets $\Lambda_i^K(R)$ for all $1\leq i\leq d.$ 
% Intuitively, considering the higher coefficients in our context should provide more information. 
We prove that 
 $R$ (need not be unmixed) is generalized Cohen-Macaulay if and only if $\Lambda_i^K(R)$ (equivalently $\delta_i^K(R)$) are finite for 
all $1\leq i \leq d-1$ (Theorem \ref{thm-main-gcm2}). In \cite[Theorem 1.1]{goto-ozeki} authors proved that  $R$ is generalized Cohen-Macaulay if and only if $\Lambda_i(R)$ is finite for all $1 \leq i \leq d.$ We improve their result and extend it to modules. More precisely,
 we show that $M$ (resp. $M/\h^0_\m(M)$) is generalized Cohen-Macaulay (resp. Buchsbaum) if and only if $|\Lambda_i(M)| < \infty$ (resp. $|\Lambda_i(M)|=1$) for 
all $1 \leq i \leq r-\depth M$ (Theorems \ref{thm-generalize-for-modules} and \ref{thm:charcOfBuchs}).

% % S. Goto and K. Ozeki proved a similar looking but weaker result for the coefficients $e_i(Q,R)$. They proved that $R$ is generalized Cohen-Macaulay if and only if $\Lambda_i(R)$ is 
% %  finite for all $1 \leq i \leq d$ \cite[Theorem 1.1]{goto-ozeki}. In the course of proving our main results, we also improve their result and extend it to modules. More precisely,
% %  we show that $M$ (resp. $M/\h^0_\m(M)$) is generalized Cohen-Macaulay (resp. Buchsbaum) if and only if $|\Lambda_i(M)(R)| < \infty$ (resp. $|\Lambda_i(M)(R)|=1$) for 
% % all $1 \leq i \leq r-\depth M,$ see Theorems \ref{thm-generalize-for-modules} and \ref{thm:charcOfBuchs}.

In section \ref{section-finiteness-in-general}, we consider the set 
\begin{eqnarray*}
% % \Delta_R(M)&=&\{e_1(I,M)~|~I \text{ is an $\m$-primary ideal of }R\},\\
\Delta^K(R)&=&\{g_1^K(I)~|~I \text{ is an $\m$-primary ideal of }R\}.
\end{eqnarray*} In \cite[Theorem 1.1]{KT}, authors proved that the set $ \{e_1(I,R)~|~I \text{ is an $\m$-primary ideal of }R\}$ is finite if and only if $d=1$ and $R/\h^0_\m(R)$ is analytically unramified.  
% % We examine the finiteness of the set $\Delta^K(R).$ 
We prove that $\Delta^K(R)$ is finite if and only if $d=1$ and $R/\h^0_\m(R)$ is analytically unramified (Theorem \ref{theorem:finiteness-of-g1}). 
% They proved that  In section \ref{section-finiteness-in-general}, we prove that $\Delta_R(M)$ is finite if and only if $d=1$ and $R^\prime/\ann_{R^\prime}(M^\prime)$ is analytically unramified, where $R^\prime=R/\h_\m^0(R)$ and $M^\prime=M/\h_{\m}^0(M)$ (Theorem \ref{theorem:finHilCoeffOfNoethModules}). As a consequence we give a criterion for the set $\Delta^K(R)$ to be finite .

We gather preliminary results needed in section \ref{section:prelim}.

Few words about proofs. Considering $K$ as an $R$-module, we get a relation between $g_i^K(Q)$ and $e_i(Q,K)$ which shows 
that $|\Lambda_i^K(R)|=|\Lambda_i(K)|$ (See \eqref{eqn-g1}). This suggests that results on the finiteness properties of the set $\Lambda_i(M),$ for any finitely generated module $M$, are useful to study the similar properties of $\Lambda_i^K(R).$ This method is used in order to study the finiteness of the set $\Lambda_i^K(R)$ in this paper. This method depends on the module theoretic properties of $K$ and is used to study the finiteness of the set $\Lambda_i^K(R)$ in this paper. 
% To be able to use this method, we extend a number of results known for rings to modules through out the paper. 
% However this may not be the efficient
% method as it depends on the module theoretic properties of $K.$
% % However, this is not an efficient method as it needs or gives an information about $K$ (not $R$).    

We refer \cite{matsumura} and \cite{bruns-herzog} for undefined terms.

\section{Preliminaries}
\label{section:prelim}
In this section we prove some preliminary results needed in the subsequent sections. We first note a relation between the Hilbert coefficients and the fiber coefficients.
\begin{remark}
 \begin{enumerate}
  \item 
 Let $d \geq 1.$ Since $\ell_R(R/KQ^n)=\ell_R(R/K)+\ell_R(K/Q^nK),$ for all $n\in \ints,$ $P_K(Q,n)=\ell_R(R/K)+P(Q,n,K).$ Thus comparing the coefficients of both sides, we get \begin{eqnarray}                                                                                                                                                                                                      
\label{eqn-g1}
  g_0^K(Q)=e_0(Q,K) \mbox{ and }  g_i^K(Q)=\begin{cases}
                e_i(Q,K) & \text{if }\ i\neq d\\
                e_d(Q,K)+(-1)^d\ell_R(R/K) & \text{if }\ i=d                .
 \end{cases}
               \end{eqnarray} 
% Thus for $d=1,$ $g_1(Q)=  e_1(Q,K)-\ell_R(R/K)=-\ell_R(H^0_{\m}(K))-\ell_R(R/K).$\\           
\item Since $\ell_R(R/KQ^n)=\ell_R(R/Q^n)+\ell_R(Q^n/KQ^n)$ for all integers $n$, we have 
$P_K(Q,n)=P(Q,n,R)+P(F,n)$ for all integers $n.$ Thus comparing the coefficients of both sides, we get
\begin{eqnarray}
\label{equation relating g0 and f0} g_0^K(Q)&=&e_0(Q,R) \mbox{~~~and}\\ 
 f_i^K(Q)&=&e_{i+1}(Q,R)-g_{i+1}^K(Q)+e_i(Q,R)-g_i^K(Q) \text{~~for $0\leq i\leq d-1$}.\label{eqn-fi-gi-ei}
 \end{eqnarray}
 \end{enumerate}
 \end{remark} 

We now recall few definitions. A module $M$ of dimension $r$ is said to be {\it generalized Cohen-Macaulay} if $\h^i_\m(M)$ has finite length for all 
$0 \leq i\leq r-1,$ where $\h^i_\m(M)$ denotes the $i$-th local cohomology module of $M$ with support in $\m.$ For a parameter ideal $Q,$ set 
$$I(Q;M):= \ell_R(M/QM)-e_0(Q,M)\mbox{ and }I(M):=\sup \{I(Q;M):Q \mbox{ is a parameter ideal for }M\}.$$   
It is well-known that $M$ is generalized Cohen-Macaulay if and only if $I(M)< \infty.$ In this case
\begin{eqnarray} \label{eqn:I(M)}
 I(M)=\sum_{i=0}^{r-1}\binom{r-1}{i} \ell_R(\h^i_\m(M)).
 \end{eqnarray}

 We refer \cite{cst} and \cite{trung-gcm} for details.
 
\begin{definition}
 \begin{enumerate}
% %  \item A system of parameters 
  \item A parameter ideal $Q$ for $M$ is said to be {\it standard for $M$} if $I(Q;M)=I(M).$ An ideal $I$ with $\ell_R(M/IM)<\infty$ is said to be {\it$M$-standard ideal} 
%   if $Q$ is a standard parameter ideal for all $Q \subseteq I.$
if every parameter ideal for $M$ contained in $I$ is standard for $M$.
  \item An $R$-module $M$ is said to be {\it Buchsbaum} if every parameter ideal for $M$ is standard.
 \end{enumerate}
\end{definition}

In the following lemma we relate the properties of $R$ and $K$ as an $R$-module.
\begin{lemma}\label{prop-basic1-gcm}
 Let $(R,\m)$ be a Noetherian local ring of dimension $d\geq 1$ and $K$ an $\m$-primary ideal of $R$. Then 
 \begin{enumerate}
%   \item \label{prop-basic1-gcm-1} As an $R$-module, $\dim K=d$.
  \item \label{prop-basic1-gcm-2} $R$ is a generalized Cohen-Macaulay ring if and only if $K$ is a  generalized Cohen-Macaulay $R$-module. 
  \item \label{prop-basic1-gcm-3} Suppose $\depth R >0$ and $K$ is a Buchsbaum $R$-module. Then $R$ is a Buchsbaum ring. 
  \item \label{prop-basic1-gcm-4} If $R$ is Buchsbaum then $\m$ is a Buchsbaum $R$-module.
 \end{enumerate}
\end{lemma}
\begin{proof}
\eqref{prop-basic1-gcm-2}:
 Consider the exact sequence
 $$0\to K\to R\to R/K\to 0.$$
This induces the exact sequence
\begin{equation} \label{eqn:LongExactSeq}
0\to \h_{\m}^0(K)\to \h_{\m}^0(R) \to \h_{\m}^0(R/K)=R/K\to\h_{\m}^1(K)\to \h_{\m}^1(R)\to 0 
\end{equation}
and isomorphisms 
\begin{equation} \label{eqn:Iso}
\h_{\m}^i(K)\cong \h_{\m}^i(R) \text{ for } 2\leq i\leq d. 
\end{equation}
This shows that $\h^i_\m(K)$ has finite length if and only if $\h^i_\m(R)$ has finite length for $1 \leq i \leq d-1.$ Hence the assertion follows.
 
\eqref{prop-basic1-gcm-3}:
% % We show that every parameter ideal of $R$ is standard.
 Let $Q=(x_1,\ldots,x_d)$ be an arbitrary parameter ideal of $R.$ We show that $Q$ is standard for $R$.
 Since $\depth R >0,$ \eqref{eqn:LongExactSeq} gives an exact sequence
  \begin{equation} \label{eqn:ShortExactSeq}
   0\to R/K\to \h_{\m}^1(K)\to \h_{\m}^1(R)\to 0. 
  \end{equation}
  Thus 
  \begin{eqnarray*}
   &&\ell_R(R/QK)\\
   &=& \ell_R(R/K) + \ell_R(K/QK)\\
   &=&\ell_R(R/K) + \sum_{i=1}^{d-1}\binom{d-1}{i}\ell_R(\h_{\m}^i(K)) + e_0(Q,K) \hspace{3.5cm} \mbox{ (as $K$ is Buchsbaum)}\\
   &=&\ell_R(R/K) + \sum_{i=1}^{d-1}\binom{d-1}{i}\ell_R(\h_{\m}^i(R))+(d-1)\ell_R(R/K) + e_0(Q,K)  \hspace{0.5cm}
   \mbox{ (from \eqref{eqn:Iso} and \eqref{eqn:ShortExactSeq})}\\
   &=&e_0(Q,R)+I(R)+d\ell_R(R/K) \hspace{5.5cm} \mbox{ (since $e_0(Q,K)=e_0(Q,R)$)}.
  \end{eqnarray*}
Hence, by \cite[Corollary 4.9]{trung-gcm}, $Q$ is a standard parameter ideal of $R.$

\eqref{prop-basic1-gcm-4}: 
Let $Q=(x_1,\ldots,x_d)$ be a parameter ideal for $\m.$ We have 
\begin{eqnarray*}
 I(Q;\m) &=&\ell_R(\m/Q\m)-e_0(Q,\m)\\
 &=&\ell_R(R/Q)+\ell_R(Q/Q\m)-\ell_R(R/\m)-e_0(Q,R) \hspace{2cm}(\mbox{as }e_0(Q,\m)=e_0(Q,R))\\
&=& I(Q;R)+d-1\\
&=& I(R)+d-1  \hspace{7cm} (\mbox{since $Q$ is standard for $R$})
 \end{eqnarray*}
which is independent of $Q.$ Hence $\m$ is Buchsbaum. 
\end{proof}
% % %   
% % % \sal{Indeed, we show that if $Q$ is standard for $K$, then $Q$ is standard for $R$. Thus we can modify the statement of the theorem \ref{thm-buchsbaum-constant}
% % % to incorporate the corollary in it to save space. Anyway we do not use (1) implies (2) of this Theorem anywhere. We may put Corollary \ref{corr-constant-value-of-gi}
% % % nd its previous theorem in this section instead of Preliminaries.}

\section{The set $\Lambda_1^K(R)$} \label{SetOfg1}
In this section we study the finiteness of the set $\Lambda_1^K(R).$ We give an equivalent criterion for the set $\Lambda_1^K(R)$ to be finite in an unmixed local ring (Theorem \ref{GCM}). 
% This leads to a characterization for generalized Cohen-Macaulayness. 
We  also consider the problem when $g_1^K(Q)$ is independent of $Q.$ For $K=\m,$ we give a characterization for $|\Lambda_1^K(R)|=1$ in an unmixed local ring and obtain partial results for arbitrary $K$ (Theorem \ref{thm: Buchsbaum}).
% The main aim of this section is to give necessary and sufficient conditions for $|\Lambda_1^K(R)|=1.$ 
% In \cite{vanishing1, vanishing2}, authors have shown that in an unmixed module $M$ the set $\Lambda_1(M)$ is finite if and only if $M$ is a generalized Cohen-Macaulay module. 
% We obtain a similar result for the finiteness of $\Lambda_1^K(R)$. 
% In \cite[Theorem 5.4]{vanishing2}, authors have shown that $|\Lambda_1(M)|=1$ if and only if $M$ is a Buchsbaum $R$-module provided $R$ is an unmixed local ring. We give an example to show that $\Lambda_1^K(R)$ may not be singleton for an arbitrary $\m$-primary ideal $K$ in a Buchsbaum ring. We obtain an equivalent criterion for $|\Lambda_1^K(R)|=1$ for $K=\m$. 
% as well as for the fiber coefficients $f_0(Q)$ .

Recall that a module $M$ is said to be {\it unmixed} if $\dim \widehat{R}/\p=\dim M$ for all $\p \in \ass_{\widehat{R}} (\widehat{M}),$ where $\widehat{M}$ 
denotes the $\m$-adic completion of $M.$ In the following proposition we give bounds on $g_1^K(Q)$ in generalized Cohen-Macaulay local rings which are 
independent of $Q.$ Consequently, we give an equivalent criterion for the finiteness of $\Lambda_1^K(R)$ in terms of $K$ in an unmixed local ring.

% % In the following proposition we give a necessary and sufficient condition for the finiteness of the set $\Lambda_1^K(R)$ in an unmixed local ring.

\begin{proposition} \label{proposition:g1}
  Let $(R,\m)$ be a Noetherian local ring of dimension $d\geq 2$ and $K$ an $\m$-primary ideal of $R.$ 
  \begin{enumerate}
   \item \label{proposition:g1-p1}Suppose $R$ is generalized Cohen-Macaulay. Then the following assertions hold.
   \begin{enumerate}
    \item \label{proposition:g1-p1-a} For any parameter ideal $Q$ of $R,$ 
      $ -\mathop\sum\limits_{i=1}^{d-1}\binom{d-2}{i-1}\ell_R(\h_{\m}^i(R))- \ell_R(R/K) \leq g_1^K(Q) \leq 0 .$ In particular, 
      $\Lambda_1^K(R)$ is finite.
   \item \label{proposition:g1-p1-b} If $Q$ is a standard parameter ideal for $K,$ then
   $$g_i^K(Q)= \begin{cases}
           (-1)^i\mathop\sum\limits_{j=1}^{d-i}\binom{d-i-1}{j-1}\ell_R(\h_{\m}^j(K)) & \mbox{ if } 1\leq i\leq d-1\\
           (-1)^d (\ell_R(\h_{\m}^0(K)) + \ell_R(R/K))& \mbox{ if }i=d.
     \end{cases}
     $$
   \end{enumerate}
   \item \label{proposition:g1-p2}Suppose $R$ is an unmixed local ring. Then $K$ is a generalized Cohen-Macaulay (resp. Buchsbaum) $R$-module if and only if $|\Lambda_1^K(R)|<\infty$ (resp. $|\Lambda_1^K(R)|=1$).
  \end{enumerate}
 \end{proposition}
 \begin{proof}
 \eqref{proposition:g1-p1}: Since $R$ is generalized Cohen-Macaulay, by Lemma \ref{prop-basic1-gcm}\eqref{prop-basic1-gcm-2}, $K$ is a generalized Cohen-Macaulay $R$-module. Hence, by \cite[p. 47]{vanishing2}, $-\mathop\sum\limits_{i=1}^{d-1}\binom{d-2}{i-1}\ell_R(\h_{\m}^i(K)) \leq e_1(Q,K) \leq 0.$ 
 Using \eqref{eqn:LongExactSeq} and  \eqref{eqn:Iso}, we get that $-\ell_R(\h_{\m}^1(R))-\ell_R(R/K) \leq -\ell_R(\h_{\m}^1(K))$ and $\ell_R(\h_\m^i(R))=\ell_R(\h_\m^i(K))$ for all $2 \leq i \leq d-1,$ respectively.  
Thus, $-\mathop\sum\limits_{i=1}^{d-1}\binom{d-2}{i-1}\ell_R(\h_{\m}^i(R)) -\ell (R/K)\leq e_1(Q,K) \leq 0.$ 
Now \eqref{proposition:g1-p1-a} follows from \eqref{eqn-g1}.
 
 If $Q$ is a standard parameter ideal for $K$ then, by \cite[Corollary 4.2]{trung-gcm}, 
 $$e_i(Q,K)= (-1)^i\mathop\sum\limits_{j=0}^{d-i}\binom{d-i-1}{j-1}\ell_R(\h_{\m}^j(K)). $$
Hence \eqref{proposition:g1-p1-b} follows from  \eqref{eqn-g1}.
 
 \eqref{proposition:g1-p2}: Since $R$ is unmixed, $K$ is an unmixed $R$-module. Also, from \eqref{eqn-g1},  $|\Lambda_1^K(R)|=|\Lambda_1(K)|.$ Hence the result follows from \cite[Theorem 4.5 and Theorem 5.4]{vanishing2}.
 \end{proof}

% \begin{enumerate} 
% \item By Proposition \ref{prop-basic1-gcm}, $K$ is a generalized Cohen-Macaulay $R$-module. Hence $\Lambda_1(K)$ is a finite set, from \cite[Section 4]{vanishing2}. From \eqref{eqn-g1}, we get that $|\Lambda_1^K(R)|=|\Lambda_1(K)|$. Thus $\Lambda_1^K(R)$ is a finite set.

% \item 
% \end{enumerate}

% \end{proof}
The following theorem provides an equivalent criterion for an unmixed local ring $R$ to be generalized Cohen-Macaulay in terms of the set 
$\Lambda_1^K(R)$. 

% In the following theorem we give an equivalent criterion for $R$ to be generalized Cohen-Macaulay in terms of the set $\Lambda_1^K(R)$ provided $R$ is unmixed. 
% Furthermore, in Theorem \ref{thm-main2}, we remove unmixedness from the hypothesis. 

\begin{theorem} \label{GCM}
Let $(R,\m)$ be an unmixed local ring of dimension $d\geq 2$ and $K$ an $\m$-primary ideal of $R$. Then the following conditions are equivalent:
\begin{enumerate}
 \item \label{GCM1} $R$ is generalized Cohen-Macaulay;
 \item \label{GCM2} $\Lambda_1^K(R)$ is a finite set;
 \item \label{GCM3} $\delta_1^K(R)$ is a finite set;
 \end{enumerate}
\end{theorem}
\begin{proof}
% % We prove \eqref{GCM1} $\Rightarrow$ \eqref{GCM2} $\Rightarrow$ \eqref{GCM3} $\Rightarrow$ \eqref{GCM1}.\\
\eqref{GCM1} $\Rightarrow$ \eqref{GCM2}: Follows from Proposition \ref{proposition:g1}\eqref{proposition:g1-p1-a}.\\
\eqref{GCM2} $\Rightarrow$ \eqref{GCM3}: Since $\delta_1^K(R) \subseteq \Lambda_1^K(R),$ the assertion follows.\\
\eqref{GCM3} $\Rightarrow$ \eqref{GCM1}: 
We may assume that $R$ is complete. Since $\delta_1^K(R)$ is a finite set,  by \eqref{eqn-g1}, the set $S(K):=\{e_1(Q,K)| Q \text{ is a parameter ideal of $R$ and } Q\subseteq K\}$ is finite. 
% Now we recall \cite[Lemma 4.1]{vanishing2}.
%  For a finitely generated module $M$ of dimension $r$ and an integer $t\geq 1$, define 
%  $$\Lambda_1(M)(t):=\{e_1(Q,M)| Q=(x_1,\ldots,x_r)\subseteq \m^t \text{ is parameter ideal for $M$ and } x_1,\ldots,x_r \text{ is d-sequence for }M\}.$$
%  If $R$ is homomorphic image of Gorenstein ring and $\ass_R(M)=\assh_R(M)$, then $R$ contains a system of parameters $x_1,\ldots,x_r$ of $M$ such that
%  $x_1^{n_1},\ldots,x_r^{n_r}$ is a d-sequence for $M$ for all integers $n_1,\ldots,n_r\geq 1$. Therefore $\Lambda_1(M)(t)\neq \phi$ for $t\geq 1$.
%  Authors have proved in \cite[Lemma 4.1]{vanishing2} that if $\Lambda_1(M)(t)(M)\neq \phi$ is a finite set for some $t\geq 1$, then $M$ is gcm.
Let $l$ be an integer such that $\m^l\subseteq K.$ Then the set 
 $$\{e_1(Q,K)| Q  =(x_1,\ldots,x_d)\subseteq \m^l \mbox{ is a parameter ideal of }R \text{ which is a $d$-sequence for }K\}\subseteq S(K)$$
 is finite. Since $R$ is unmixed, $K$ is an unmixed $R$-module. Therefore by \cite[Lemma 4.1]{vanishing2}, $K$ is a generalized Cohen-Macaulay $R$-module. Hence by Proposition \ref{prop-basic1-gcm}\eqref{prop-basic1-gcm-2}, $R$ is generalized Cohen-Macaulay.
%  To achieve $\ass_R(K)=\assh_R(K)$, we can move to $\hat{R}$ and use the fact that $K\hat{R}=\hat{K}$. $\hat{R}$ is unmixed implies $\hat{K}$ is unmixed.
%  The conclusion would be $\hat{K}$ is gcm and hence $\hat{R}$ is gcm. This gives $R$ is gcm.
\end{proof}

For a finitely generated $R$-module $M,$ we set $\assh_R M = \{\p \in \ass_R M|\dim R/\p=\dim M\}.$ Let 
$(0_M)=\mathop\bigcap\limits_{\p \in \ass_R M} M(\p)$ be a primary decomposition of $(0_M)$ in $M,$ where $M(\p)$ is a $\p$-primary submodule of $M$ for 
each $\p \in \ass_R M.$ The $R$-submodule $U_M(0):= \mathop\bigcap\limits_{\p \in \assh_R M} M(\p)$ is called the {\it unmixed component} of $M.$ 
% % \sal{We need following modified version of  \cite[Lemma 4.3]{vanishing2}.}
 
 In order to prove the next theorem we need a modified version of \cite[Lemma 4.3]{vanishing2}.
 \begin{lemma}\label{lemma-0}
% %  \rm\cite[Lemma 4.3]{vanishing2}
  Let $(R,\m)$ be a Noetherian local ring and $M$ a finitely generated $R$-module with $\dim M=r\geq 2$. 
  Let $K$ be an $\m$-primary ideal of $R.$ Assume that there exists an integer $t\geq 0$ such that $e_1(Q,M)\geq -t$ for every parameter 
  ideal $Q\subseteq K$ for $M.$ Then $\dim U_M(0)\leq r-2$.
 \end{lemma}
\begin{proof}
Let $U=U_M(0)$ and $T=M/U.$ Since $U_\p=0$ for all $\p \in \assh_R(M),$ $\dim U \leq r-1.$ Suppose $\dim U=r-1.$ Choose a system of parameters $(x_1,\ldots,x_r)$ for $M$ such that $x_rU=0.$ Since $\m^l\subseteq K$ for some 
integer $l\geq 1,$ $Q=(x_1^s,\ldots,x_r^s)\subseteq K$ for all $s\geq l.$ Let $s>\max\{l,t\}.$ Consider the exact sequence
$$0\to U/(Q^{n+1}M\cap U) \to M/Q^{n+1}M\to T/Q^{n+1}T\to 0.$$
This gives $$\ell_R(M/Q^{n+1}M)=\ell_R(T/Q^{n+1}T)+\ell_R(U/(Q^{n+1}M\cap U)).$$
By Artin-Rees Lemma there exists an integer $k\geq 0$ such that $Q^{n}M\cap U=Q^{n-k}(Q^kM\cap U)$ for all $n\geq k.$ Let $U^{\prime}=Q^kM\cap U$ and 
$q=(x_1^s,\ldots,x_{r-1}^s).$ Since $Q^{n-k}U^{\prime}=q^{n-k}U^{\prime}$ for all $n \geq k,$ we get 
$$\ell_R(M/Q^{n+1}M)=\ell_R(T/Q^{n+1}T)+\ell_R(U^{\prime}/q^{n+1-k}U^{\prime})+\ell_R(U/U^{\prime}) \mbox{ for all }n\geq k.$$
This implies that $-t\leq e_1(Q,M)=e_1(Q,T)-e_0(q,U^{\prime}).$
 Since $e_0(q,U^{\prime})=e_0(q,U)$ and $e_1(Q,T)\leq 0$ by \cite[Theorem 3.6]{msv}, we get
$$s\leq s^{r-1}e_0((x_1,\ldots,x_{r-1}),U)=e_0(q,U)=e_1(Q,T)-e_1(Q,M)\leq t,$$
which is a contradiction. Thus $\dim U\leq r-2.$
\end{proof}

In the following theorem we give equivalent conditions for the finiteness of the set $\Lambda_1^K(R)$ in any Noetherian local ring.

 \begin{theorem}\label{thm-main2}
 Let $(R,\m)$ be a Noetherian local ring of dimension $d\geq 2.$ We set $U=U_{\widehat{R}}(0).$ Then the following conditions are equivalent:
 \begin{enumerate}
  \item \label{eqn-thm-main2-2} $\dim_{\widehat{R}} U\leq d-2$ and $\widehat{R}/U$ is a generalized Cohen-Macaulay ring;
  \item \label{eqn-thm-main2-1} $\Lambda_1^K(R)$ is a finite set;
  \item \label{thm-main2-part2} $\delta_1^K(R)$ is a finite set.
    \end{enumerate}
 When this is the case, we have 
 $$-\mathop\sum\limits_{i=1}^{d-1}\binom{d-2}{i-1}\ell_R(\h_{\m}^i(\widehat{R}/U))-\ell_R(R/K) \leq  g_1^{K}(Q)\leq 0.  $$
 for every parameter ideal $Q$ of $R.$
\end{theorem}
\begin{proof}
 We may assume that $R$ is complete. \\
%  Let $R^\prime=R/U.$\\
\eqref{eqn-thm-main2-2} $\Rightarrow$ \eqref{eqn-thm-main2-1}: Since $R/U$ is a generalized Cohen-Macaulay ring, by 
Proposition \ref{proposition:g1}\eqref{proposition:g1-p1-a}, the set $\Lambda_1^{KR/U}(R/U)$ is finite. Since $g_1^K(R)=g_1^{K R/U}(R/U),$  
by \cite[Lemma 3.6]{saloni-1}, the set $\Lambda_1^K(R)$ is finite. (Note that we do not need $Q\subseteq K$ in \cite[Lemma 3.6]{saloni-1}.)\\
 \noindent
 \eqref{eqn-thm-main2-1} $\Rightarrow$ \eqref{thm-main2-part2}: Since $\delta_1^K(R) \subseteq \Lambda_1^K(R), $ the assertion follows. \\
  \noindent  \eqref{thm-main2-part2} $\Rightarrow$ \eqref{eqn-thm-main2-2}: From \eqref{eqn-g1}, $e_1(Q,K)=g_1^K(Q).$ Thus $\delta_1^K(R)$ is finite implies 
  that there exists an integer $t\geq 0$ such that $e_1(Q,K)\geq -t$ for every parameter ideal $Q\subseteq K.$ 
  Hence, by Lemma \ref{lemma-0}, $\dim U_K(0)\leq d-2.$ Note that $U \cap K =U_K(0).$ 
%  and $K/U_K(0)$ is generalized Cohen-Macaulay $R$-module by \cite[Theorem 4.5]{vanishing2} .
 Since $\dim U= \text{max}\{\dim(U\cap K), \dim (U/(U\cap K)) \}$ and $ \dim (U/(U\cap K))=0,$ 
%  Since $U/(U\cap K)\cong (U+K)/K$ and $K$ is an $\m$-primary ideal, $\dim (U/(U \cap K))=0.$ 
%  using Proposition \sm{add} \ref{remark-prim-dec}\eqref{remark-prim-dec-4}, 
 $\dim U= \dim (U\cap K)= \dim (U_K(0))\leq d-2.$ Hence, by \cite[Lemma 3.6]{saloni-1}, $g_1^K(Q)=g_1^{K R/U}(QR/U).$ Using \cite[Remark 4.4]{vanishing2} we conclude that $\delta_1^{K R/U}(R/U)$ is 
 finite. Hence, by Theorem \ref{GCM}, $R/U$ is generalized Cohen-Macaulay. 
 
 The last assertion follows from Proposition \ref{proposition:g1}\eqref{proposition:g1-p1-a}.
 %  $K/U_K(0)\cong K/(U\cap K)\cong (K+U)/U$. 
%  By Proposition \ref{prop-basic1-gcm}\eqref{prop-basic1-gcm-2}, 
%  $(K+U)/U$ is generalized Cohen-Macaulay implies that $R/U$ is generalized Cohen-Macaulay.
% By Proposition \ref{remark-prim-dec}\eqref{remark-prim-dec-4}, $\dim U_K(0)=\dim (U\cap K)\leq \dim U\leq d-2$. $R/U$ is generalized
%  Cohen-Macaulay implies $(K+U)/U\cong K/U_K(0)$ is generalized Cohen-Macaulay module by Proposition \ref{prop-basic1-gcm}\eqref{prop-basic1-gcm-1}.
%  Therefore $\Delta_1(K)$ is finite by \cite[Theorem 4.5]{vanishing2}. Using \eqref{eqn-g1}, we get that $\Delta_1^K(R)$ is finite. 
 \end{proof}
 
In the following theorem we give a sufficient condition for $R$ to be Buchsbaum.

\begin{theorem} \label{thm: Buchsbaum}
  Let $(R,\m)$ be a Noetherian local ring of dimension $d \geq 2$ and $K$ an $\m$-primary ideal of $R.$ Then the following assertions hold.
  \begin{enumerate}
   \item \label{thm: Buchsbaum-p1}
    Suppose $R$ is unmixed and $|\Lambda_1^K(R)|=1.$  Then $R$ is Buchsbaum. Further, $|\Lambda_i^K(R)|=1$ for all $1\leq i\leq d.$
    \item \label{thm: Buchsbaum-p2} If $R$ is Buchsbaum then $|\Lambda_1^\m(R)|=1.$ 
  \end{enumerate}
  \end{theorem}
  \begin{proof}
  \eqref{thm: Buchsbaum-p1}: By Proposition \ref{proposition:g1}\eqref{proposition:g1-p2}, we get that $K$ is a Buchsbaum $R$-module.  
 Hence by Lemma \ref{prop-basic1-gcm}\eqref{prop-basic1-gcm-3},
 $R$ is a Buchsbaum ring. Since every parameter ideal of $R$ is standard for $K,$ by Proposition \ref{proposition:g1}\eqref{proposition:g1-p1-b}, $|\Lambda_i^K(R)|=1$ for all $1\leq i\leq d.$ 
 
 \noindent
\eqref{thm: Buchsbaum-p2}: By Lemma \ref{prop-basic1-gcm}\eqref{prop-basic1-gcm-4}, $\m$ is a Buchsbaum $R$-module. Thus every parameter ideal $Q$ of $R$ is standard for $\m$. Now by 
Proposition \ref{proposition:g1}\eqref{proposition:g1-p1-b}, $|\Lambda_1^\m(R)|=1.$ 

% %  \item By Lemma \ref{prop-basic1-gcm}\eqref{prop-basic1-gcm-4}, $\m$ is a Buchsbaum $R$-module. Hence, by \cite[Corollary 4.2]{trung-gcm}, $e_1(Q,\m)= -\mathop\sum\limits_{i=1}^{d-1}\binom{d-2}{i-1}\ell_R(\h_{\m}^i(\m))$ for every parameter ideal $Q$ of $R.$ Since $e_1(Q,\m)=g_1^\m(Q)$, from \eqref{eqn-g1}, $|\Lambda_1^\m(R)|=1.$
 \end{proof}
% % %  \sal{For $K=\m$, remove unmixedness and write the general statement}.
\begin{theorem}\label{thm-main3}
Let $(R,\m)$ be a Noetherian local ring of dimension $d\geq 2.$ Let $U=U_{\widehat{R}}(0).$ Then the following conditions are equivalent:
  \begin{enumerate}
   \item \label{thm-main3-a} $\dim_{\widehat{R}} U \leq d-2$ and $\widehat{R}/U$ is Buchsbaum; 
   \item \label{thm-main3-b} $|\Lambda_1^{\m}(R)|=1$. 
  \end{enumerate}
\end{theorem}
\begin{proof}
We may assume that $R$ is complete. 

 \eqref{thm-main3-a}$\Rightarrow$ \eqref{thm-main3-b}: By \cite[Lemma 3.6]{saloni-1} and Theorem \ref{thm: Buchsbaum}\eqref{thm: Buchsbaum-p2}, we get $|\Lambda_1^{\m}(R)|=|\Lambda_1^{\m}(R/U)|=1$.

 \eqref{thm-main3-b}$\Rightarrow$ \eqref{thm-main3-a}: Since $|\Lambda_1^{\m}(R)|=1,$ by \eqref{eqn-g1},  $|\Lambda_1(\m)|=1.$ 
%  Let $(0)=\mathop\bigcap\limits_{\p \in \ass(R)} Q(\p)$ be a minimal primary decomposition of $(0)$ in $R,$ where $Q(\p)$ is a $\p$-primary ideal of $R.$ Then $(0_\m)=\mathop\bigcap\limits_{\p \in \ass(R)} Q(\p)$ and for $\p \neq \m,$ $Q(\p)$ is a $\p$-primary submodule of $\m.$ This implies that $\ass_R(\m) \setminus \{\m\} =\ass(R) \setminus \{\m\}.$ 
 Since $U_\m(0)=U,$ by \cite[Theorem 5.5]{vanishing2}, $\dim U=\dim U_{\m}(0)\leq d-2.$ Thus $g_1^{\m}(Q)=g_1^{\m R/U}(QR/U)$ by \cite[Lemma 3.6]{saloni-1}. Hence $|\Lambda_1^{\m}(R/U)|=|\Lambda_1^{\m}(R)|=1.$ Therefore, by Theorem \ref{thm: Buchsbaum}\eqref{thm: Buchsbaum-p1}, $R/U$ is Buchsbaum. 
\end{proof}

% % For an arbitrary $\m$-primary ideal $K,$ $\Lambda_1^K(R)$ need not be singleton. We give an example and explanation for this in the following discussion. 
 We discuss below that for an arbitrary $\m$-primary ideal $K$ in a Buchsbaum local ring $R,$ $\Lambda_1^K(R)$ need not be singleton.

\begin{discussionbox}
 \label{disc:Buchs}
Suppose $(R,\m)$ is a Buchsbaum local ring of dimension $d \geq 2$ and $K$ is an $\m$-primary ideal of $R.$ Suppose $|\Lambda_1^K(R)|=1.$ 
Then, from \eqref{eqn-g1}, $|\Lambda_1(K)|=1.$ Further assume that $R$ is unmixed. Then, by \cite[Theorem 5.4]{vanishing2}, $K$ is a Buchsbaum $R$-module. Let $Q$ be an arbitrary parameter ideal of $R.$ Since 
$$\ell_R(R/Q^{n+1}K)=\ell_R(R/K)+\ell_R(K/Q^{n+1}K) \hspace{0.2cm} \mbox{ for all }n$$
and $Q$ is standard for $K, $ using \cite[Corollary 4.2]{trung-gcm}, we get
\begin{eqnarray*}
 \ell_R(R/Q^{n+1}K)=\ell_R(R/K)+\binom{n+d}{d}e_0(Q,K)+\sum_{i=1}^d\sum_{j=0}^{d-i} \binom{n+d-i}{d-i}\binom{d-i-1}{j-1}\ell_R(\h^j_\m(K))
\end{eqnarray*}
for all $n \geq 0.$ Putting $n=0$ and using \eqref{eqn:LongExactSeq} and \eqref{eqn:Iso}, we get 
\begin{equation}\label{eqn:EquForKBuch}
 \ell_R(R/QK)=e_0(Q,R)+I(R)+d \ell (R/K).
 \end{equation}
Also,
\begin{eqnarray}
\nonumber \ell_R(R/QK)&=&\ell_R(R/Q)+\ell_R(Q/QK)\\
\label{eqn:EquAlw} &=&e_0(Q,R)+I(R)+\ell_R(Q/QK).
\end{eqnarray}
Comparing \eqref{eqn:EquForKBuch} and \eqref{eqn:EquAlw}, we get $\ell_R(Q/QK)=d\ell_R(R/K)$ for every parameter ideal $Q$ of $R.$ This need not be true even in regular local rings. 
% For example, let $R=k[|x,y|]$ and $K=(x,y)^2.$ Then for $Q=(x,y)$, $\ell_R(Q/QK) \neq 2\ell_R(R/K).$

However, we expect that $|\delta_1^K(R)|=1$ for an arbitrary $\m$-primary ideal $K$ in a Buchsbaum local ring. We have neither a proof nor a counter-example for this statement.
\end{discussionbox}

\begin{remark}
 \begin{enumerate}
  \item 
 Suppose $R$ is a generalized Cohen-Macaulay local ring. Then, by \cite[Section 4]{vanishing2}, $\Lambda_1(R)$ is finite. By Proposition \ref{proposition:g1}\eqref{proposition:g1-p1}, $\Lambda_1^K(R)$ is 
 finite. Hence, from \eqref{eqn-fi-gi-ei}, the set $\{f_0^K(Q)~|~Q \mbox{ is a parameter ideal of } R\}$ is finite.
 \item Suppose $R$ is Buchsbaum. Then, by \cite[Section 5]{vanishing2}, $\Lambda_1(R)$ is singleton. By  Theorem \ref{thm: Buchsbaum}\eqref{thm: Buchsbaum-p2}, $|\Lambda_1^\m(R)|=1.$ Hence, from \eqref{eqn-fi-gi-ei}, the set $\{f_0^\m(Q)~|~Q \mbox{ is a parameter ideal of } R\}$ is singleton.
  \end{enumerate}
\end{remark}

\section{The set $\Lambda_i^K(R)$} \label{SetOfgi}
In this section we give a necessary and sufficient condition for $\Lambda_i^K(R)$ to be finite for all $1 \leq i \leq d$ (Theorem \ref{thm-main-gcm2}). For this 
purpose we improve a result of Goto and Ozeki \cite[Theorem 1.1]{goto-ozeki} and generalize it for modules (Theorem \ref{thm-generalize-for-modules}). 
We also obtain an equivalent criterion for $|\Lambda_i(M)|=1$ for all $1 \leq i \leq r-\depth M$ (Theorem \ref{thm:charcOfBuchs}). As a consequence we 
obtain a necessary condition for $|\Lambda_i^K(R)|=1$ for all $1 \leq i \leq d-1$ (Theorem \ref{thm-main-Buchs}).
% Let us now recall the following result  \cite[Corollary 4]{cuong-long-truong2015} 
% which provides a bound on the Castelnuovo-Mumford regularity of the associated graded module $G_Q(M)$ of parameters of generalized Cohen-Macaulay modules.
% 
% \begin{theorem}\rm\cite[Corollary 4]{cuong-long-truong2015}\label{thm-regularity-for-modules}
%  Let $M$ be a generalized Cohen-Macaulay module. Then, there exists a constant $C$ such that $reg(G_Q(M))\leq C$ for all parameter ideals $Q$ of $M$.
% \end{theorem}

We need few lemmas in order to prove the finiteness of $\Lambda_i^K(R)$ in a generalized Cohen-Macaulay local ring. First we recall the following lemma 
from \cite{trung-gcm}.

\begin{lemma}
\label{lemma-for-uniform-bounds-trung} \rm \cite[Lemma 1.7]{trung-gcm} Let $M$ be a generalized Cohen-Macaulay module of dimension $r$ and $Q=(x_1,\ldots,x_r)$ a parameter ideal for $M.$ Then $I(M/x_1M ) \leq I(M).$
\end{lemma}

In the next lemma we give a bound on the function $ \ell_R(M/Q^{n+1}M)$ in terms of $e_0(Q,M)$ and $I(M)$  
for a generalized Cohen-Macaulay module. A similar upper bound is given for $ \ell_R(R/Q^{n+1})$ in \cite[Lemma 1.1]{linh-trung2006}. A better lower bound is given for $ \ell_R(R/Q^{n+1})$ in terms of $e_0(Q,R)$ in \cite[Theorem 1.1]{hh}.  
% % \sm{A better bound is known for the difference $\ell_R(M/Q^{n+1}M)-e_0(Q,M)\binom{n+d}{d}.$ In the following lemma we give a simple proof for a weak lower bound.}  

\begin{lemma}\label{lemma-for-uniform-bounds-3}
Let $M$ be a generalized Cohen-Macaulay module of dimension $r>0$ and $Q$ a parameter ideal for $M.$ Then for all $n\geq 0,$
$$-r\binom{n+r-1}{r-1}I(M)\leq \ell_R(M/Q^{n+1}M)-e_0(Q,M)\binom{n+r}{r}\leq \binom{n+r-1}{r-1}I(M).$$
\end{lemma}
\begin{proof}
We apply induction on $r.$ Let $r=1.$ Set $W:=\h^0_\m(M)$ and $M^\prime:=M/\h_{\m}^0(M).$ Then $M^\prime$ is a Cohen-Macaulay $R$-module and $e_0(Q,M^\prime)=e_0(Q,M)$. Hence for all $n\geq 0,$ 
 \begin{eqnarray*}
  \ell_R(M/Q^{n+1}M)&=&\ell_R(M^\prime/Q^{n+1}M^\prime)+\ell_R(W/(W\cap Q^{n+1}M))\\
  &=&e_0(Q,M)(n+1)+\ell_R(W/(W\cap Q^{n+1}M)).
  \end{eqnarray*}
 Therefore for all $n \geq 0,$
 \begin{eqnarray*}
 0\leq \ell_R(M/Q^{n+1}M) -  e_0(Q,M)(n+1)  \leq & \ell_R(W)=I(M).
 \end{eqnarray*}
Thus the result is true for $r=1. $ Now let $r>1$ and $Q=(x_1,\ldots,x_r)$ be a parameter ideal for $M.$ We put $\overline{M}=M/x_1M.$ 
Since $M$ is a generalized Cohen-Macaulay module, by  \cite[Claim 1 in p. 351]{cuong-long-truong2015}, we have 
\begin{eqnarray}\label{eqn:bound_on_colon}
0 \leq \ell_R((Q^{n+1}M:x_1)/Q^nM)\leq \binom{n+r-2}{r-2}I(M) \mbox{ for all }n\geq 0. 
\end{eqnarray}
By \cite[Lemma 1.2]{trung-gcm}, $\dim R/\p=r-i$ for all $\p\in \ass (M/(x_1,\ldots,x_i)M)\setminus \{\m\}$ and $i=1,\ldots,r-1.$ Thus, using \cite[Corollary 4.8]{ab}, we get  
$$e_0(Q,\ov{M})=\ell_R(M/QM)-\ell_R((x_1,\ldots,x_{r-1})M:x_r/x_{r-1}M)=e_0(Q,M).$$
% Since $e_0(Q,M)=e_0(Q,\overline{M}),$ 
By Lemma \ref{lemma-for-uniform-bounds-trung}, $ I(\overline{M})\leq I(M).$
Hence applying induction hypothesis, we get 
\begin{equation}\label{eqn:bound_by_induction}-(r-1)\binom{n+r-2}{r-2}I(M)\leq \ell_R(\overline{M}/Q^{n+1}\overline{M}) - e_0(Q,M)\binom{n+r-1}{r-1} 
\leq \binom{n+r-2}{r-2}I(M)
 \end{equation}
for all $n \geq 0.$ Considering the exact sequence 
$$0\to (Q^{t+1}M:x_1)/Q^tM\to M/Q^tM\xrightarrow{~x_1} M/Q^{t+1}M\to \overline{M}/Q^{t+1}\overline{M}\to 0, $$
we get \begin{eqnarray} 
\label{eqn-exact sequence}
        \ell_R(Q^tM/Q^{t+1}M)&=&\ell_R(\overline{M}/Q^{t+1}\overline{M})-\ell_R((Q^{t+1}M:x_1)/Q^tM) \hspace{0.2cm}\mbox{ for all } t \geq 0.
                \end{eqnarray}
Hence, using \eqref{eqn:bound_on_colon} and \eqref{eqn:bound_by_induction}, we get 
\begin{equation}\label{eqn:exact_seq}
-r \binom{t+r-2}{r-2}I(M)\leq \ell_R(Q^tM/Q^{t+1}M) -  e_0(Q,M)\binom{t+r-1}{r-1} \leq \binom{t+r-2}{r-2}I(M). 
\end{equation}
%         &\leq & \ell_R(M^\prime/Q^{t+1}M^\prime)\\
%         &\leq &e_0(Q,M^\prime)\binom{n+t}{t}+\binom{n+t-1}{t-1}I(M^\prime) \mbox{ by induction hypothesis}\\
%         & \leq & e_0(Q,M) \binom{n+t}{t} + \binom{n+t-1}{t-1}I(M) \mbox{ since } e_0(Q,M)=e_0(Q,M^\prime) \mbox{ and } I(M^\prime)\leq I(M).\\
%          \end{eqnarray*}
Since $\ell_R(M/Q^{n+1}M)=\mathop\sum\limits_{t=0}^n\ell_R(Q^tM/Q^{t+1}M),$
using \eqref{eqn:exact_seq}, we get 
$$-r\binom{n+r-1}{r-1}I(M)\leq \ell_R(M/Q^{n+1}M)-e_0(Q,M)\binom{n+r}{r}\leq \binom{n+r-1}{r-1}I(M).$$
\end{proof}
% In \cite[Corollary 1.4]{linh-trung2006}, authors proved that 
% if $R$ is generalized Cohen-Macaulay ring and $Q$ is a parameter ideal of $R$, then for $n\geq 0$, 
% $\ell_R((Q^{n+1}:x_1)/Q^n)\leq \binom{n+d-2}{d-2}I(R).$ Similar result is true for generalized Cohen-Macaulay modules. See \cite[Lemma 7, Claim 1]{cuong-long-truong2015} for proof. 
% \begin{lemma}\label{lemma-for-uniform-bounds-4}\rm\cite[Lemma 7, Claim 1]{cuong-long-truong2015}
% Let $M$ be a generalized Cohen-Macaulay module of dimension $d\geq 2$ and $Q=(x_1,\ldots,x_d)$ be a parameter ideal of $M$.
% Then for $n\geq 0$ and $m\geq 1$,
% $$\ell_R((Q^{n+1}M:x_d^m)/Q^nM)\leq \binom{n+d-2}{d-2}I(M).$$
% \end{lemma}

% \begin{remark}\label{remark-for-uniform-bounds-1}
%  Using Lemmas \ref{lemma-for-uniform-bounds-3} and \ref{lemma-for-uniform-bounds-2}, we get that 
%  $$\big{|}\ell_R(M/Q^{n+1}M)-e_0(Q,M)\binom{n+d}{d}\big{|}\leq d\binom{n+d-1}{d-1}I(M).$$
% \end{remark}

Let $G_Q(M)=\mathop\bigoplus\limits_{n \geq 0} Q^nM/Q^{n+1}M$ be the associated graded module of $M$ with respect to $Q.$ Let $\mathcal M=\mathop\bigoplus\limits_{n \geq 1} [G_Q(R)]_n$ and  
$$a_i(G_Q(M))= \sup\{n \in \ints : [\h^i_{\mathcal M}(G_Q(M))]_n \neq 0\}.$$  
Recall that 
$$ \reg(G_Q(M))=\sup \{a_i(G_Q(M))+i:i \in \ints\} $$ 
is the {\it Castelnuovo-Mumford regularity} of the graded module $G_Q(M).$ We need the following lemma 
in order to obtain uniform bounds on the coefficients $e_i(Q,M)$ in terms of $\reg(G_Q(M)).$ 
We skip the proof of this as it is similar to \cite[Lemma 2.3]{goto-ozeki}.

\begin{lemma}\label{lemma-for-uniform-bounds-1}
 Let $M$ be a finitely generated module of dimension $r>0$ and $Q$ a parameter ideal for $M.$  
 \begin{enumerate}
  \item \label{lemma-for-uniform-bounds-1-a} Let $M^\prime=M/\h_{\m}^0(M)$. Then $\reg(G_Q(M))\geq \reg(G_Q(M^\prime))$.
  \item \label{lemma-for-uniform-bounds-1-b} Assume that $r\geq 2$ and $x\in Q$ is superficial for $M$ with respect to $Q.$ Let $\ov{Q}=Q/(x)$ in $\ov{R}=R/(x)$
  and $\ov{M}=M/xM$. Then $\reg(G_Q(M))\geq \reg(G_{\ov{Q}}(\ov{M}))$.
   \end{enumerate}
\end{lemma}

We use a method similar to \cite[Theorem 2.2]{goto-ozeki} to prove the following theorem.

\begin{theorem}\label{thm-uniform-bounds-for-modules}
 Let $M$ be a generalized Cohen-Macaulay module of dimension $r>0$ and $Q$ a parameter ideal for $M.$ Put $\kappa=\reg(G_Q(M)).$ Then
 \begin{enumerate}
  \item \label{uniform-bounds-a} $|e_1(Q,M)|\leq I(M).$
  \item \label{uniform-bounds-b} $|e_i(Q,M)|\leq (r+1) \cdot 2^{i-2}  (\kappa+1)^{i-1}I(M)$ for $2\leq i\leq r$.
 \end{enumerate}
\end{theorem}
\begin{proof}
We may assume that the residue field $R/\m$ is infinite.
 We use induction on $r.$ Let $r=1.$ Then by \cite[Proposition 3.1]{msv}, $e_1(Q,M)=-\ell_R(\h_{\m}^0(M))$ for all parameter ideals $Q$ for $M.$ 
Hence $|e_1(Q,M)|= \ell_R(\h_{\m}^0(M))=I(M).$ Thus the assertion is true in this case. 

Let $r\geq 2.$ We may assume that $\depth M>0.$ In fact, let $M^\prime=M/\h_{\m}^0(M)$ and assume that 
the assertion holds for $M^\prime.$ Set $\kappa^\prime=\reg(G_Q(M^\prime)).$ By \cite[Proposition 2.3]{RV}, 
\begin{eqnarray}\label{equation:ei-M-&-M'}
 e_i(Q,M)=\begin{cases}
           e_i(Q,M^\prime) & \mbox{ if }i \neq r\\
           e_r(Q,M^\prime)+(-1)^r \ell_R(\h^0_\m(M)) & \mbox{ if }i=r.
         \end{cases}
\end{eqnarray}
Hence 
\begin{eqnarray*}
 |e_1(Q,M)|&\leq &|e_1(Q,M^\prime)|+\ell_R(\h_{\m}^0(M)) \leq I(M^\prime) +  \ell_R(\h_{\m}^0(M)) = I(M)
\end{eqnarray*}
and for $2 \leq i \leq r,$
% $|e_i(Q,M^\prime)|\leq (r+1) 2^{i-2}(\bar{r}+1)^{i-1}I(M^\prime)$ for $2\leq i\leq r$ where $\bar{r}=\reg(G_Q(M^\prime))$. 
 \begin{eqnarray*}
  |e_i(Q,M)|&\leq &|e_i(Q,M^\prime)|+\ell_R(\h_{\m}^0(M))\\
  &\leq & (r+1)\cdot2^{i-2}(\kappa^\prime+1)^{i-1}I(M^\prime) +\ell_R(\h_{\m}^0(M))\\
  &\leq & (r+1)\cdot2^{i-2}(\kappa^\prime+1)^{i-1}\left(\sum_{i=1}^{r-1}\binom{r-1}{i}\ell_R(\h_{\m}^i(M^\prime))+\ell_R(\h_{\m}^0(M)) \right)\\
%   &\leq& (r+1)2^{i-2}(r+1)^{i-1}\big{(}\sum_{i=1}^{r-1}\binom{r-1}{i}\ell_R(H_{\m}^i(M))\big{)}+\ell_R(H_{\m}^0(M))\\
%   &&\hspace{6cm} 
  &\leq& (r+1)\cdot2^{i-2}(\kappa+1)^{i-1}\left(\sum_{i=0}^{r-1}\binom{r-1}{i}\ell_R(\h_{\m}^i(M))\right) 
  \hspace{1cm}\text{[by Lemma } \ref{lemma-for-uniform-bounds-1}\eqref{lemma-for-uniform-bounds-1-a}]\\
 &=& (r+1)\cdot2^{i-2}(\kappa+1)^{i-1}I(M).
 \end{eqnarray*}

 Let $Q=(x_1,\ldots,x_r)$ be such that $x_1$ is superficial for $M$ with respect to $Q$. Put $\overline{M}=M/x_1M$ and $\bar{\kappa}=\reg(G_{Q}(\ov{M}).$ Then, using 
 induction hypothesis and Lemmas \ref{lemma-for-uniform-bounds-trung} and \ref{lemma-for-uniform-bounds-1}\eqref{lemma-for-uniform-bounds-1-b}, we get  
 \begin{eqnarray}\label{equation:bound-for-e1}
  |e_1(Q,M)|=|e_1(Q,\ov{M})|&\leq& I(\ov{M})\leq I(M) 
 \end{eqnarray}
and for $2 \leq i \leq r-1,$ 
\begin{eqnarray}\label{equation:bound-for-induction-step}
 |e_i(Q,M)|=|e_i(Q,\ov{M})|&\leq& r\cdot2^{i-2}(\bar{\kappa}+1)^{i-1}I(\ov{M})\leq (r+1)\cdot2^{i-2}(\kappa+1)^{i-1}I(M).
\end{eqnarray}
Let $i=r.$ By \cite[Theorem 4.4.3]{bruns-herzog}, for all $t>\kappa$
$$\ell_R(Q^t M/Q^{t+1}M)=\sum_{i=0}^{r-1}(-1)^ie_i(Q,M)\binom{t+r-1-i}{r-1-i}$$
and for all $t\geq \bar{\kappa}$,
$$\ell_R(\ov{M}/Q^{t+1}\ov{M})=\sum_{i=0}^{r-1}(-1)^ie_i(Q,\ov{M})\binom{t+r-1-i}{r-1-i}.$$
Since $\kappa\geq \bar{\kappa}$ by Lemma \ref{lemma-for-uniform-bounds-1}\eqref{lemma-for-uniform-bounds-1-b}, we get $\ell_R(Q^tM/Q^{t+1}M)=\ell_R(\ov{M}/Q^{t+1}\ov{M})$ for all $t>\kappa$. 
% On the 
% other hand, from \eqref{eqn-exact sequence}, for all $t\geq 0$, we have
%   $$\ell_R(Q^t M/Q^{t+1}M)=\ell_R(\bar{M}/Q^{t+1}\bar{M})-\ell_R((Q^{t+1}M:x_1)/Q^tM)).$$ 
Hence, using \eqref{eqn-exact sequence}, we get $(Q^{t+1}M:x_1)=Q^tM$ for all $t>\kappa$. 
% \begin{eqnarray}
%  \ell_R(M/Q^{n+1}M)&=&\sum_{t=0}^n\ell_R(Q^tM/Q^{t+1}M)\nonumber\\
%  &=& \sum_{t=0}^n\ell_R(\bar{M}/Q^{t+1}\bar{M})-\sum_{t=0}^n\ell_R((Q^{t+1}M:x_1)/Q^tM)\nonumber\\
%  && \hspace{6cm} [\text{by } \ref{eqn-induction-step}]\nonumber\\
% &=&\sum_{t=0}^n\ell_R(\bar{M}/Q^{t+1}\bar{M})-\sum_{t=0}^r\ell_R((Q^{t+1}M:x_1)/Q^tM)\label{eqn-e*}\\
% \end{eqnarray}
Since $\ell_R(M/Q^{n+1}M)=\mathop\sum\limits_{i=0}^{r}(-1)^ie_i(Q,M)\binom{n+r-i}{r-i}$ for all $n\geq \kappa$ by \cite[Theorem 4.4.3]{bruns-herzog}, we get 
\begin{eqnarray} 
 &&(-1)^re_r(Q,M) \nonumber\\
 &=& \ell_R(M/Q^{n+1}M)-\sum_{i=0}^{r-1}(-1)^ie_i(Q,M)\binom{n+r-i}{r-i}\nonumber\\
 &=&\sum_{t=0}^n \ell_R({Q^tM}/Q^{t+1}{M}) - \sum_{i=0}^{r-1}(-1)^ie_i(Q,M)\binom{n+r-i}{r-i}\nonumber\\
 &=& \sum_{t=0}^n\ell_R(\ov{M}/Q^{t+1}\ov{M})-\sum_{t=0}^n\ell_R((Q^{t+1}M:x_1)/Q^tM)-\sum_{t=0}^{n}\sum_{i=0}^{r-1}(-1)^ie_i(Q,\ov{M})\binom{t+r-1-i}{r-1-i}\nonumber\\
 &&\hspace{10cm} [\text{from } \eqref{eqn-exact sequence} ]\nonumber\\
%  &=&\sum_{t=0}^n\ell_R(\ov{M}/Q^{t+1}\ov{M})-\sum_{t=0}^r\ell_R((Q^{t+1}M:x_1)/Q^tM)-\sum_{t=0}^{n}\sum_{i=0}^{r-1}(-1)^ie_i(Q,\ov{M})\binom{t+r-1-i}{r-1-i}\nonumber\\
&=&\sum_{t=0}^n\left(\ell_R(\ov{M}/Q^{t+1}\ov{M})-\sum_{i=0}^{r-1}(-1)^ie_i(Q,\ov{M})\binom{t+r-1-i}{r-1-i}\right)-\sum_{t=0}^{\kappa}\ell_R((Q^{t+1}M:x_1)/Q^tM)\nonumber\\
&=&\sum_{t=0}^{\bar{\kappa}}\left(\ell_R(\ov{M}/Q^{t+1}\ov{M})-\sum_{i=0}^{r-1}(-1)^ie_i(Q,\ov{M})\binom{t+r-1-i}{r-1-i}\right)-\sum_{t=0}^\kappa\ell_R((Q^{t+1}M:x_1)/Q^tM)\nonumber\\
&=&\sum_{t=0}^{\bar{\kappa}}\left(\ell_R(\ov{M}/Q^{t+1}\ov{M})- e_0(Q,\ov{M})\binom{t+r-1}{r-1}\right)-\sum_{t=0}^{\bar{\kappa}} \sum_{i=1}^{r-1}(-1)^ie_i(Q,\ov{M})\binom{t+r-1-i}{r-1-i}\nonumber\\
 & & -\sum_{t=0}^{\kappa}\ell_R((Q^{t+1}M:x_1)/Q^tM).\label{eqn-e*2} \nonumber
 \end{eqnarray}
This implies that
%  Now by induction hypothesis, 
%  \begin{eqnarray*}
%   |e_1(Q,M)|=|e_1(Q,\bar{M}|&\leq& I(\bar{M})\leq I(M)\text{ and }\nonumber\\
%   |e_i(Q,M)|=|e_i(Q,\bar{M}|&\leq& r.2^{i-2}(\bar{r}+1)^{i-1}I(\bar{M})\leq (r+1)2^{i-2}(r+1)^{i-1}I(M)\nonumber
%  \end{eqnarray*}
% for $2\leq i\leq r-1$.  
\begin{eqnarray*}
&& |e_r(Q,M)|\\
&\leq& \sum_{t=0}^{\bar{\kappa}}\left|\ell_R(\ov{M}/Q^{t+1}\ov{M})-e_0(Q,\ov{M})\binom{t+r-1}{r-1}\right|+ \sum_{t=0}^{\bar{\kappa}} |e_1(Q,M)|\binom{t+r-2}{r-2}\\
& & +\sum_{t=0}^{\bar{\kappa}} \sum_{i=2}^{r-1}|e_i(Q,M)|\binom{t+r-1-i}{r-1-i} +\sum_{t=0}^{\kappa}\ell_R((Q^{t+1}M:x_1)/Q^tM)\\
&\leq& \sum_{t=0}^{\bar{\kappa}}(r-1)\binom{t+r-2}{r-2}I(\ov{M})+\sum_{t=0}^{\bar{\kappa}} \binom{t+r-2}{r-2}I(M)\\
& &+\sum_{t=0}^{\bar{\kappa}} \sum_{i=2}^{r-1}(r+1)\cdot2^{i-2}(\kappa+1)^{i-1}\binom{t+r-1-i}{r-1-i}I(M) + \sum_{t=0}^{\kappa}\binom{t+r-2}{r-2}I(M)\\
& & \hspace{6cm} [\text{using Lemma } \ref{lemma-for-uniform-bounds-3} \mbox{ and Equations }\eqref{eqn:bound_on_colon},\eqref{equation:bound-for-e1} \mbox{ and }\eqref{equation:bound-for-induction-step}]\\ 
&=& (r-1)\binom{\bar{\kappa}+r-1}{r-1}I(\ov{M})+\binom{\bar{\kappa}+r-1}{r-1}I(M)\\
& & +\sum_{i=2}^{r-1}(r+1) \cdot 2^{i-2}(\kappa+1)^{i-1}\binom{\bar{\kappa}+r-i}{r-i}I(M)+\binom{\kappa+r-1}{r-1}I(M)\\
&\leq & (r+1) \binom{\kappa+r-1}{r-1}I(M)+\sum_{i=2}^{r-1}(r+1)\cdot2^{i-2}(\kappa+1)^{i-1}\binom{\kappa+r-i}{r-i}I(M)\\
& & \hspace{6cm } [ \text{using Lemmas } \ref{lemma-for-uniform-bounds-trung} \mbox{ and }\ref{lemma-for-uniform-bounds-1} \eqref{lemma-for-uniform-bounds-1-b}]\\
&\leq & (r+1)(\kappa+1)^{r-1}I(M)+ \sum_{i=2}^{r-1}(r+1)\cdot2^{i-2}(\kappa+1)^{r-1}I(M)\\
& & \hspace{6cm} [ \text{since } \binom{m+n}{n}\leq (m+1)^n \text{ for all integers } n\geq 0]\\
&=& (r+1)(\kappa+1)^{r-1}I(M) \left(1+\sum_{i=2}^{r-1}2^{i-2}\right)\\
&=& (r+1)\cdot2^{r-2}(\kappa+1)^{r-1}I(M) \hspace{1.5cm} [\mbox{since }\sum_{i=2}^{r-1}2^{i-2}=2^{r-2}-1].
\end{eqnarray*}
\end{proof}

In the following lemma we give a necessary condition for the finiteness of the set $\Lambda_i(M)$ for all $1 \leq i \leq k,$ where $k$ is a fixed integer such that $1 \leq k \leq r.$ The proof given here is motivated by \cite[Theorem 1.1]{goto-ozeki}.

\begin{lemma} \label{lemma:finOfFewSets}
 Let $(R,\m)$ be a Noetherian local ring, $K$ an $\m$-primary ideal of $R$ and $M$ a finitely generated $R$-module of dimension $r\geq 2.$ For a fixed $1 \leq k \leq r,$ assume that 
 $$\{e_i(Q,M):Q\mbox{ is a parameter ideal for }M \mbox{ and } Q \subseteq K \}$$ 
 is a finite set for all $1 \leq i \leq k.$ Then $\ell_R(\h_\m^{r-i}(M))<\infty$ for all $1 \leq i \leq k.$ 
 
In particular, if $\Lambda_i(M)$ is finite for all $1\leq i\leq k,$ then $\ell_R(\h_\m^{r-i}(M))<\infty$ for all $1 \leq i \leq k.$
\end{lemma}
\begin{proof}
% % We may assume that $R$ is complete. Let $U:=U_M(0)$ and $N=M/U.$ If $U=0$ then $M$ is unmixed and hence by \cite[Theorem 4.5]{vanishing2}, $M$ is a generalized Cohen-Macaulay module. Thus $\ell_R(\h_\m^{r-i}(R))<\infty$ for all $1 \leq i \leq r.$
We may assume that $R$ is complete. Let $l$ be an integer such that $\m^l\subseteq K$. Let $U=U_M(0)$ and $N=M/U.$ If $U=0$ then $M$ is unmixed and 
the set $\{e_1(Q,M):Q=(x_1,\ldots,x_r)\subseteq \m^l \mbox{ and } x_1,\ldots,x_r \mbox{ is a $d$-sequence for } M\}$ is finite. 
Hence by \cite[Lemma 4.1]{vanishing2}, $M$ is a generalized Cohen-Macaulay module. Thus $\ell_R(\h_\m^{r-i}(R))<\infty$ for all $1 \leq i \leq r.$

Assume that $U \neq 0.$ 
% % By \cite[Theorem 4.5]{vanishing2}, $N$ is a generalized Cohen-Macaulay module and $\dim U\leq r-2.$ 
By Lemma \ref{lemma-0}, $\dim U\leq r-2.$ Hence by \cite[Lemma 3.3]{vanishing2}, $e_1(Q,M)=e_1(Q,N).$ Thus the 
set $\{e_1(Q,N):Q\mbox{ is a parameter ideal for }M \mbox{ and } Q \subseteq K \}$ is finite. By \cite[Remark 4.4]{vanishing2}, the set $\{e_1(Q,N):Q\mbox{ is a parameter ideal for }N \mbox{ and } Q \subseteq K \}$ is also finite. Hence, by $U=0$ case, $N$ is generalized Cohen-Macaulay.
We now show that $t:=\dim U \leq r-(k+1).$  
We may assume that $t\geq 1$. Let $x_1,\ldots,x_r$ be a system of parameters for $M$ such that $(x_{t+1},\ldots,x_r)U=0.$ Since $N$ is a generalized Cohen-Macaulay module, by \cite[Lemma 1.5]{trung-gcm},
there exists an integer $l_1\geq 1$ such that $\m^{l_1}$ is a standard ideal for $N.$ 
% % Choose an integer $l_2\geq 1$ such that $\m^{l_2} \subseteq K.$ 
Let $l_0=\max\{l_1,l\}.$ Then $\m^{l_0} \subseteq K$ is a standard ideal. Let $n\geq l_0$ and $Q=(x_1^n,\ldots,x_r^n).$ Then by \cite[Corollary 4.2]{trung-gcm},  
$$e_{r-t}(Q,N)=(-1)^{r-t}\mathop\sum\limits_{j=1}^t\binom{t-1}{j-1}\ell_R(\h_{\m}^j(N)).$$ 
We have 
 $$\ell_R(M/Q^{n+1}M)=\ell_R(N/Q^{n+1}N)+\ell_R(U/(Q^{n+1}M\cap U)) \mbox{ for all }n\geq 0.$$ 
Since the filtration $\{Q^{n+1}M\cap U\}$ is a good $Q$-filtration of $U$ (see \cite[page 1]{RV} for the definition of good $Q$-filtration), 
% the function $\ell_R(U/(Q^{n+1}M\cap U))$ is a polynomial of degree $t$ for $n\gg 0.$ Hence   
$$\ell_R(U/(Q^{n+1}M\cap U))=\mathop\sum\limits_{i=0}^t(-1)^is_i(Q,U)\binom{n+t-i}{t-i}$$
for some integers $s_i(Q,U)$ with $s_0(Q,U)=e_0(Q,U).$ This implies that for $n\gg 0$,
\begin{eqnarray*}
\ell_R(M/Q^{n+1}M)
% &=&\mathop\sum\limits_{i=0}^r(-1)^ie_i(Q,M)\binom{n+r-i}{r-i}\\
&=& \mathop\sum\limits_{i=0}^r(-1)^ie_i(Q,N)\binom{n+r-i}{r-i}+\mathop\sum\limits_{i=0}^t(-1)^is_i(Q,U)\binom{n+t-i}{t-i}.\label{eqn-M-N-U-relation}
\end{eqnarray*}
Therefore for $n\geq l_0,$
\begin{eqnarray*}
 (-1)^{r-t}e_{r-t}(Q,M)&=&(-1)^{r-t}e_{r-t}(Q,N)+e_0(Q,U)\\
 &=& \mathop\sum\limits_{j=1}^t\binom{t-1}{j-1}\ell_R(\h_{\m}^j(N))+n^te_0((x_1,\ldots,x_t),U)\\
& \geq & n^t.
\end{eqnarray*}
Thus $\Lambda_{r-t}(M)$ is not finite which implies that $r-t\geq k+1$. Thus $ t\leq r-(k+1).$ Consequently, $\h_{\m}^i(U)=0$ for all $i \geq r-k.$ Hence $\h_{\m}^i(M) \simeq \h_{\m}^i(N)$ has finite length for all
$r-k \leq i\leq r-1.$
\end{proof}

Next, we improve a result of Goto and Ozeki \cite[Theorem 1.1]{goto-ozeki} and generalize it for modules. In order to prove this we recall the following result from \cite{cuong-long-truong2015}.

\begin{theorem} \rm\cite[Corollary 4]{cuong-long-truong2015} \label{thm:bound on regularity}
 Let $M$ be a generalized Cohen-Macaulay module. Then, there exists a constant $C$ such that $\reg(G_Q(M))\leq C$ for all parameter ideals $Q$ for $M.$
\end{theorem}

\begin{theorem}\label{thm-generalize-for-modules}
 Let $(R,\m)$ be a Noetherian local ring and $M$ a finitely generated $R$-module of dimension $r\geq 2.$ 
 Then the following conditions are equivalent: 
 \begin{enumerate}
  \item \label{module-a} $M$ is a generalized Cohen-Macaulay module;
  \item \label{module-b} The set $\Lambda_i(M)$ is finite for all $1\leq i\leq r;$
  \item \label{module-c} The set $\Lambda_i(M)$ is finite for all $1\leq i\leq r-\depth M.$
 \end{enumerate}
\end{theorem}
\begin{proof}
\eqref{module-a} $\Rightarrow $ \eqref{module-b}: Follows from Theorems \ref{thm-uniform-bounds-for-modules} and \ref{thm:bound on regularity}.\\
 \eqref{module-b} $\Rightarrow $ \eqref{module-c}: This is clear.\\
 \eqref{module-c} $\Rightarrow $ \eqref{module-a}:
 Follows from Lemma \ref{lemma:finOfFewSets}.
  \end{proof}

  We now discuss an example from \cite{goto-ozeki} which illustrates the significance of the finiteness of $\Lambda_i(M)$ for $i=\dim M-\depth M.$ 

\begin{example} \cite[Example 3.5]{goto-ozeki} Let $(R,\mathfrak{n})$ be a regular local ring of dimension $d\geq 2$ and $X_1,\ldots,X_d$ a regular system of parameters of $R.$ We put $\mathfrak{p}=(X_1,\ldots,X_{d-1})$ and $D=R/\mathfrak{p}.$ Let $A=R\ltimes D$ be the idealization of $D$ over $R.$ Then $A$ is a Noetherian local ring with the maximal ideal $\m=\mathfrak{n}\times D, \dim A=d$ and $\depth A=1.$ 
By \cite[Example 3.5]{goto-ozeki}
\begin{align*}\label{eqn-22-2}
 \Lambda_i(A)=\begin{cases}
 	                      \{n~|~0< n\in\ints\} &\text{ if }\ i=0 \\
	  \{0\}& \text{ if } 1\leq i\leq d \text{ and } i\neq d-1\\ 
                \{(-1)^{d-1}n~|~0<n\in\ints\} & \text{ if } i=d-1
               \end{cases}
 \end{align*}
 and 
$\h_{\m}^1(A)(\cong \h_{\mathfrak{n}}^1(D))$ is not a 
finitely generated $A$-module. Hence $A$ is not generalized Cohen-Macaulay.
\end{example}
  
As a consequence of Theorem \ref{thm-generalize-for-modules} we obtain a characterization of generalized Cohen-Macaulay rings in terms of the coefficients $g_i^K(Q).$

\begin{theorem}\label{thm-main-gcm2}
 Let $(R,\m)$ be a Noetherian local ring of dimension $d\geq 2$ and $K$ an $\m$-primary ideal of $R.$ Then the following conditions are equivalent:
 \begin{enumerate}
  \item \label{thm-main-gcm2-part1} $R$ is generalized Cohen-Macaulay;
  \item \label{thm-main-gcm2-part2} $\Lambda^K_i(R)$ is finite for all $1\leq i\leq d;$
  \item \label{thm-main-gcm2-part3} $\Lambda^K_i(R)$ is finite for all $1\leq i\leq d-1;$
  \item \label{thm-main-gcm2-part4} $\delta_i^K(R)$ is finite for all $1\leq i\leq d-1.$
 \end{enumerate}
\end{theorem}
\begin{proof}
By Lemma \ref{prop-basic1-gcm}\eqref{prop-basic1-gcm-2}, \eqref{thm-main-gcm2-part1} is equivalent to the generalized Cohen-Macaulayness of $K.$ From \eqref{eqn-g1}, 
$|\Lambda_i(K)|=|\Lambda^K_i(R)|$ for all $1 \leq i \leq d.$ 
% Thus \eqref{thm-main-gcm2-part1} $\Rightarrow$ \eqref{thm-main-gcm2-part2} $\Rightarrow$ \eqref{thm-main-gcm2-part3} 
% which implies that \eqref{thm-main-gcm2-part2} is equivalent to the finiteness of the set $|\Lambda_i(K).$ 
% follows 
Hence  \eqref{thm-main-gcm2-part1} $\Rightarrow$ \eqref{thm-main-gcm2-part2} follows from Theorem \ref{thm-generalize-for-modules}. The implication 
 \eqref{thm-main-gcm2-part2} $\Rightarrow$ \eqref{thm-main-gcm2-part3} $\Rightarrow$ \eqref{thm-main-gcm2-part4} is clear. We show \eqref{thm-main-gcm2-part4} $\Rightarrow$ \eqref{thm-main-gcm2-part1}. Since $\delta_i^K(R)$ is finite, by \eqref{eqn-g1}, 
$$\{e_i(Q,K):Q\mbox{ is a parameter ideal of }R \mbox{ and }Q \subseteq K\}$$
is finite for all $1\leq i\leq d-1.$ Therefore by Lemma \ref{lemma:finOfFewSets}, $K$ is generalized Cohen-Macaulay. Thus $R$ is generalized Cohen-Macaulay by Lemma \ref{prop-basic1-gcm}\eqref{prop-basic1-gcm-2}.  
% Since $\delta_i^K(R) \subseteq \Lambda_i^K(R),$ \eqref{thm-main-gcm2-part3} $\Rightarrow$ \eqref{thm-main-gcm2-part4} is clear.
\end{proof}

In the following theorem we give a characterization for $M/\h^0_\m(M)$ to be Buchsbaum in terms of $\Lambda_i(M).$ 
See also \cite[Theorem 5.4]{vanishing2}.

\begin{theorem} \label{thm:charcOfBuchs}
Let $(R,\m)$ be a Noetherian local ring and $M$ a finitely generated $R$-module of dimension $r\geq 2.$ 
 Then the following statements are equivalent:
 \begin{enumerate}
  \item \label{thm:charcOfBuchsP1} $M/\h^0_\m(M)$ is a Buchsbaum $R$-module;
  \item \label{thm:charcOfBuchsP2} $|\Lambda_i(M)|=1$ for all $1\leq i\leq r;$
  \item \label{thm:charcOfBuchsP3} $|\Lambda_i(M)|=1$ for all $1\leq i\leq r-\depth M.$
 \end{enumerate} 
\end{theorem}
\begin{proof}
 \eqref{thm:charcOfBuchsP1} $\Rightarrow$ \eqref{thm:charcOfBuchsP2}: Let $M^\prime:=M/\h^0_\m(M).$ Since $M^\prime$ is Buchsbaum, every parameter ideal $Q$ for $M^\prime$ is standard. Hence by \cite[Corollary 4.2]{trung-gcm}, 
 $e_i(Q,M^\prime)=(-1)^i\mathop\sum\limits_{j=0}^{r-i}\binom{r-i-1}{j-1}\ell_R(\h^j_\m(M^\prime))$ for all $1\leq i\leq r.$ Thus $|\Lambda_i(M^\prime)|=1$ for all $1\leq i\leq r.$ Hence, using \eqref{equation:ei-M-&-M'}, $|\Lambda_i(M)|=|\Lambda_i(M^\prime)|=1$ for all $1 \leq i \leq r.$ \\
 \eqref{thm:charcOfBuchsP2} $\Rightarrow$ \eqref{thm:charcOfBuchsP3}: This is clear.\\
 \eqref{thm:charcOfBuchsP3} $\Rightarrow$ \eqref{thm:charcOfBuchsP1}: Let $M^\prime:=M/\h^0_\m(M).$ Since $|\Lambda_i(M)|=|\Lambda_i(M^\prime)|$ by \eqref{equation:ei-M-&-M'}, 
 $|\Lambda_i(M^\prime)|=1$ for all $1 \leq i \leq r-\depth M.$ Hence, by Theorem \ref{thm-generalize-for-modules}, $M^\prime$ is a generalized Cohen-Macaulay module $R$-module. This implies that $\widehat{M^\prime}$ is a generalized Cohen-Macaulay 
 $\widehat{R}$-module.  Since $\depth_{\widehat{R}} \widehat{M^\prime}>0,$ using \cite[Lemma 1.2]{trung-gcm}, we conclude that $M^\prime$ is an unmixed module. 
 Hence, by \cite[Theorem 5.4]{vanishing2}, $M^\prime$ is a Buchsbaum $R$-module. 
% Let $U=U_N(0).$ Then the proof of Theorem \ref{thm-generalize-for-modules} shows that $N/U$ is generalized Cohen-Macaulay and $\dim U=0.$ Hence 
% from \eqref{}, $|\Lambda_1(N/U)|=\Lambda_1(N)=1$. This implies that $N/U$ is Buchsbaum by \cite[Theorem 5.4]{vanishing2}. Now, let $Q$ be a parameter ideal of $N.$ Then 
% \begin{eqnarray*}
%  I(Q,N)&=&\ell_R(N/QN)-e_0(Q,N)\\
%  &=&\ell_R(N/QN+UN)+\ell_R(QN+UN/UN)-e_0(Q,N/U)\\
%  &=&I(Q,N/UN)+\ell_R(UN/QN \cap UN).
% \end{eqnarray*}
\end{proof}

As a consequence we give a sufficient condition for $R/\h^0_\m(R)$ to be Buchsbaum in terms of $\Lambda_i^K(R).$ 

\begin{theorem}\label{thm-main-Buchs} 
 Let $(R,\m)$ be a Noetherian local ring of dimension $d\geq 2$ and $K$ an $\m$-primary ideal of $R.$
 \begin{enumerate}
  \item \label{thm-main-Buchs-part1} Suppose $|\Lambda^K_i(R)|=1$ for all $1\leq i\leq d-1.$ Then $R/\h^0_\m(R)$ is Buchsbaum. 
  \item \label{thm-main-Buchs-part2} If $R/\h^0_\m(R)$ is Buchsbaum then $|\Lambda^\m_i(R)|=1$ for all $1\leq i\leq d.$
   \end{enumerate}
%  \begin{enumerate}
%   \item \label{thm-main-Buchs-part1} $K/H^0_\m(K)$ is Buchsbaum;
%   \item \label{thm-main-Buchs-part2} $|\Lambda^K_i(R)|=1$ for all $1\leq i\leq d;$
%   \item \label{thm-main-Buchs-part2} $|\Lambda^K_i(R)|=1$ for all $1\leq i\leq d-1.$
%  \end{enumerate}
\end{theorem}
\begin{proof} \eqref{thm-main-Buchs-part1}: From \eqref{eqn-g1}, $|\Lambda_i^K(R)|=|\Lambda_i(K)|.$ Hence taking $M=K$ in Theorem \ref{thm:charcOfBuchs}, we get that $K/\h^0_\m(K)$ is Buchsbaum. Thus, by Lemma \ref{prop-basic1-gcm}\eqref{prop-basic1-gcm-3}, $R/\h^0_\m(R)$ is Buchsbaum. 

\eqref{thm-main-Buchs-part2}: By Lemma \ref{prop-basic1-gcm}\eqref{prop-basic1-gcm-4}, $\m/\h^0_\m(\m)$ is a Buchsbaum $R$-module. Since $|\Lambda_i^\m(R)|=|\Lambda_i(\m)|,$ by Theorem \ref{thm:charcOfBuchs}, the result follows.
\end{proof}

\section{The set $\Delta^{K}(R)$}\label{section-finiteness-in-general}
For an $R$-module $M$, we set 
\begin{eqnarray*}
\Delta_R(M)&=&\{e_1(I,M)~|~I \text{ is an $\m$-primary ideal of }R\}.
% % \Delta^K(R)&=&\{g_1^K(I)~|~I \text{ is an $\m$-primary ideal of }R\}.
\end{eqnarray*}
In \cite{KT} authors gave a necessary and sufficient condition for the finiteness of the set $\Delta_R(R).$ In this section we give an equivalent criterion for 
the finiteness of the set $\Delta^K(R)$  (Theorem \ref{theorem:finiteness-of-g1}). For this purpose we first give a characterization for the set $\Delta_R(M)$ to be finite (Theorem \ref{theorem:finHilCoeffOfNoethModules}). We use a bound given by T.~Puthenpurakal, \cite[Theorem 18]{tony}, to give a sufficient
condition for the finiteness of $\Delta_R(M).$ In order to obtain a necessary condition we use ``induction''.
% % We establish a necessary condition for the finiteness of the set $\Delta_R(M)$ using ``induction.''

We need few lemmas in order to prove Theorem \ref{theorem:finHilCoeffOfNoethModules}. 
In the following lemma we show that if $\Delta_R(M)$ is finite then $\dim M=1.$ Proof given here is similar to the proof of \cite[Lemma 3.1]{KT}.

\begin{lemma}\label{lemma:finite implies d=1}
Let $(R,\mathfrak{m})$ be a Noetherian local ring and $M$ a finitely generated $R$-module of dimension $r> 0.$ Suppose $\Delta_R(M)$ is a finite set. Then $r=1.$
\end{lemma}
\begin{proof}
 Let $I$ be an $\mathfrak{m}$-primary ideal of $R$ and $k \geq 1$ an integer. We have
\begin{equation}
 \label{equation:HilbPolyforIk}
\ell_R(M/(I^k)^{n+1}M) = e_0(I^k,M)\dbinom{n+r}{r}-e_1(I^k,M)\dbinom{n+r-1}{r-1}+\ldots+(-1)^re_r(I^k,M).
\end{equation}
Also,
\begin{eqnarray}\label{equation:HilbPolyforI}
\nonumber \ell_R(M/I^{kn+k}M)& = & e_0(I,M)\dbinom{(kn+k-1)+r}{r}-e_1(I,M)\dbinom{(kn+k-1)+r-1}{r-1}\\
                     &+&\ldots+(-1)^re_r(I,M).
\end{eqnarray}
Note that 
\begin{eqnarray*}
  \dbinom{kn+k+r-1}{r} & = & k^r\dbinom{n+r}{r}+\left(k^{r-1}-k^{r}\right)\left(\frac{r-1}{2}\right)\dbinom{n+r-1}{r-1}+\mbox{lower degree terms } \mbox{ and }\\
  \dbinom{kn+k+r-2}{r-1} & = & k^{r-1}\dbinom{n+r-1}{r-1}+\mbox{lower degree terms}.
\end{eqnarray*}
Comparing \eqref{equation:HilbPolyforIk} and \eqref{equation:HilbPolyforI}, we get 
\begin{eqnarray} 
\nonumber e_0(I^k,M)&=&k^re_0(I,M) \mbox{ and } \\
\label{equation:expressionFore0Ande1}
e_1(I^k,M)&=&\frac{r-1}{2}e_0(I,M)k^r+\frac{2e_1(I,M)-(r-1)e_0(I,M)}{2}k^{r-1}.
\end{eqnarray}
Since $\Delta_R(M)$ is a finite set, the set $\{e_1(I^k,M)~|~k \geq \mbox{ 1 is an integer}\}$ is also finite. Hence using \eqref{equation:expressionFore0Ande1}, we get $r=1$.
\end{proof}

In view of Lemma \ref{lemma:finite implies d=1}, we assume that $r=1$ while examining the finiteness of the set $\Delta_R(M)$.
% % in order to establish a necessary condition for the finiteness of the set $\Delta_R(M).$ 
Now we recall the following theorem from \cite{KT} which will be used in this section.

\begin{theorem} \rm\cite[Theorem 1.1]{KT}\label{KT theorem}
Let $(R, \m)$ be a Noetherian local ring of dimension $d> 0$. Then the following conditions are equivalent:
\begin{enumerate}
\item $\Delta_R(R)$ is a finite set;
\item $d=1$ and $R/\h_\m^0(R)$ is analytically unramified.
\end{enumerate}
\end{theorem}
% % 
% % 
% % \begin{lemma} \sm{Is this new ?}
% % \label{lemma:lowerboundOfSet}
% % Let $(R, \m)$ be a Noetherian local ring and $M$ a finitely generated $R$-module of dimension one. Then 
% % $e_1(I,M) \geq -\ell_R(\h_\m^0(M))$ for every $\m$-primary ideal $I$ in $R.$
% % \end{lemma}
% % \begin{proof}
% %  Let $W=\h_\m^0(M).$ By \cite[Proposition 2.3]{RV}, 
% % % %  \begin{equation}
% % % %  \label{equation:RelHilbCoeff}
% % % % e_1(I,M) = e_1(I,M/W)-\ell_R(\h_\m^0(M)). 
% % % %  \end{equation}
% % Since $M/W$ is Cohen-Macaulay, by \cite[p.~218]{F} \sm{The statement is not exactly this. But nonnegativity follows from this.}, $e_1(I,M/W) \geq 0.$ Hence $e_1(I,M) \geq -\ell_R(\h_\m^0(M)).$
% % \end{proof}
To discuss the finiteness of $\Delta_R(M)$, we first provide bounds on this set in the following proposition. 
% % it is natural to seek bounds for $\Delta_R(M).$ In the following result we provide bounds on the set $\Delta_R(M)$. 

\begin{proposition}
 \label{proposition:SupAndInfOfSet}
 Let $(R,\m)$ be a Noetherian local ring of dimension one and $M$ a finitely generated $R$-module of dimension one. Then
 \begin{enumerate}
  \item $\inf \Delta_R(M)=-\ell_R(\h_\m^0(M)).$
  \label{proposition:InfOfSet}
  \item $\sup \Delta_R(M)\leq \ell_{R^\prime}(\overline{R^\prime}/R^\prime)\mu_{R^\prime}(M^\prime),$ where 
  $R^\prime:=R/\h_\m^0(R)$ and $M^\prime:=M/\h_\m^0(M).$ Here $\ov{R^\prime}$ denotes the integral closure of $R^\prime$ in its total ring of fractions.
  \label{proposition:SupOfSet}
  \end{enumerate}
\end{proposition}
 \begin{proof}
\eqref{proposition:InfOfSet}: Let $c=\inf \Delta_R(M).$ 
% % % By Lemma \ref{lemma:lowerboundOfSet}, $c \geq -\ell_R(\h_\m^0(M)).$ Let $I=(a)$ be a parameter ideal in $R.$ Then, by \eqref{equation:RelHilbCoeff}, 
% % $e_1(I,M)=-\ell_R(\h_\m^0(M)).$ Hence $c=-\ell_R(\h_\m^0(M)).$\\
%  By \cite[Proposition 2.3]{RV}, for every $\m$-primary ideal $I$ in $R,$
 By \eqref{equation:ei-M-&-M'}, for every $\m$-primary ideal $I$ in $R,$
\begin{equation}
 \label{equation:RelHilbCoeff}
e_1(I,M) = e_1(I,M^\prime)-\ell_R(\h_\m^0(M)). 
 \end{equation}
Since $M^\prime$ is Cohen-Macaulay, $e_1(I,M^\prime) \geq 0$ by Northcott's inequality for modules (see \cite[p.~218]{F}). Thus $ e_1(I,M) \geq -\ell_R(\h_\m^0(M))$ for every $\m$-primary ideal $I$ in $R$ which implies that $c \geq -\ell_R(\h_\m^0(M)).$ Let $Q=(x)$ be a parameter ideal for $M.$ Then, by \eqref{equation:RelHilbCoeff}, 
 $e_1(Q,M)=-\ell_R(\h_\m^0(M)).$ Hence $c=-\ell_R(\h_\m^0(M)).$\\

\eqref{proposition:SupOfSet}: Let $C=\sup \Delta_R(M).$ Note that $M^\prime$ is a maximal Cohen-Macaulay $R^\prime$-module. Hence, for every $\m$-primary ideal $I$ of $R,$ we have
\begin{align*}
 e_1(I,M)&\leq e_1(I,M^\prime), \hspace*{1.6in} \mbox{ (by \eqref{equation:RelHilbCoeff})}\\
 &=e_1(IR^\prime,M^\prime)\\
 &\leq e_1(IR^\prime,M^\prime)+e_1(IR^\prime,\syz_1^{R^\prime}(M^\prime)) \mbox{ (as }\syz_1^{R^\prime}(M^\prime) \mbox{ is a Cohen-Macaulay $R^\prime$-module)}\\
 & \leq e_1(IR^\prime,R^\prime)\mu_{R^\prime}(M^\prime) \hspace*{1in} \mbox{ (by \cite[Proposition 17]{tony})}\\
% %  &=(e_1(I)+\ell_R(\h_\m^0(R)))\mu_{R^\prime}(M^\prime), \hspace*{0.35in} \mbox{ by \eqref{equation:RelHilbCoeff} }\\
& \leq \ell_{R^\prime}(\overline{R^\prime}/R^\prime)\mu_{R^\prime}(M^\prime) \hspace*{1in}\mbox{ (by \cite[Theorem 1.2]{KT})}.
\end{align*}
Hence $C \leq \ell_{R^\prime}(\overline{R^\prime}/R^\prime)\mu_{R^\prime}(M^\prime).$
\end{proof}

In order to obtain an upper bound on the set $\Delta_R(M)$, the ring $R$ having dimension one in Proposition \ref{proposition:SupAndInfOfSet} is not a restrictive condition as we may pass to $R/\ann_R(M)$, if needed, and assume that $\dim R=1.$

\begin{proposition}
\label{proposition:relatingHilbCoeff}
Let $(R, \m)$ be a Noetherian local ring and $M$ a Cohen-Macaulay $R$-module of dimension one.  
For nonzero modules $N$ and $C,$ consider the exact sequence 
\begin{equation}
\label{givenExactSequence}
 0 \to N \to M \to C \to 0. 
\end{equation}
For an $\m$-primary ideal $I$ in $R,$ the following statements hold true.
\begin{enumerate}
\item \label{proposition:relatingHilbCoeffDim0} 
If $\dim ~C=0,$ then $e_1(I,M) \geq e_1(I,N)-\ell_R(C).$
\item  \label{proposition:relatingHilbCoeffDim1} 
If $\dim ~C=1,$ then $e_1(I,M) \geq e_1(I,N)+e_1(I,C)\geq e_1(I,N)-\ell_R(\h_\m^0(C)).$
\end{enumerate}
\end{proposition}
\begin{proof}
Tensoring \eqref{givenExactSequence} with $R/I^{n+1},$ we get an exact sequence
\begin{equation*}
  0 \to K_{I,n+1}\to \frac{N}{I^{n+1}N} \to \frac{M}{I^{n+1}M} \to \frac{C}{I^{n+1}C} \to 0, 
\end{equation*}
where $K_{I,n+1}$ (depends on $I$ and $n$) is some $R$-module of finite length. 
Therefore
\begin{equation}
\label{equation:alternatingSum}
\ell_R\left(K_{I,n+1}\right)-\ell_R\left(\frac{N}{I^{n+1}N}\right)+ \ell_R\left(\frac{M}{I^{n+1}M}\right)-\ell_R\left(\frac{C}{I^{n+1}C}\right)=0. 
\end{equation}
This implies that $\ell_R(K_{I,n+1})$ is a polynomial, for $n \gg 0,$ of degree at most one. Let 
$\ell_R(K_{I,n+1})=a_I(n+1)+b_I$, for $n \gg 0,$ where $a_I$ and $b_I$ are some integers. Since $M$ is Cohen-Macaulay, $N$ is a Cohen-Macaulay module of dimension one. Hence, by 
using \cite[Corollary 4.7.7]{bruns-herzog}, we get $a_I=0$ and hence $\ell_R(K_{I,n+1})=b_I$ for $n \gg 0.$

\ref{proposition:relatingHilbCoeff}\eqref{proposition:relatingHilbCoeffDim0}:
Suppose that $\dim ~C=0.$ Then $I^n C=0$ for $n \gg 0.$ Hence from \eqref{equation:alternatingSum}, we get that 
\begin{align*}
 e_1(I,M)=b_I+e_1(I,N)-\ell_R(C) \geq e_1(I,N)-\ell_R(C).
\end{align*}

\ref{proposition:relatingHilbCoeff}\eqref{proposition:relatingHilbCoeffDim1}: Again using \eqref{equation:alternatingSum}, we get 
\begin{align*}
e_1(I,M)&=b_I+e_1(I,N)+e_1(I,C) \\
&\geq e_1(I,N)+e_1(I,C) \\
&\geq e_1(I,N)-\ell_R(\h_\m^0(C)) \hspace*{1.5in} \mbox{(by Proposition \ref{proposition:SupAndInfOfSet}\eqref{proposition:InfOfSet})}.
\end{align*}
\end{proof}

In the following lemma we give a necessary condition for the finiteness of the set $\Delta_R(M)$ if $M$ 
is a cyclic module of dimension one. 

\begin{lemma}
\label{lemma:cyclic modules}
 Let $(R, \m)$ be a Noetherian local ring and $M=Rx$ a Cohen-Macaulay $R$-module of dimension one. Suppose $\Delta_R(M)$ is finite. Then $R/\ann_R(M)$ is analytically unramified.
\end{lemma}
\begin{proof}
 Note that $M \cong R/\ann_R(x).$ Let $B:=R/\ann_R(x).$ Since $\ell_R(B/I^n B)=\ell_{B}(B/I^nB)$ for any $\m$-primary ideal $I$ in $R,$ $e_1(I,B)=e_1(IB,B).$ Since
 every $\m B$-primary ideal in $B$ is of the form $IB$ for some $\m$-primary ideal $I$ in $R,$ finiteness of the set $\Delta_R(M)$ implies that the set $\Delta_{B}(B)$ is finite. Therefore, by Theorem \ref{KT theorem}, $B$ is analytically unramified. 
\end{proof}

Now we give an equivalent criterion for the finiteness of the set $\Delta_R(M).$

\begin{theorem}
 \label{theorem:finHilCoeffOfNoethModules}
 Let $(R, \m)$ be a Noetherian local ring and $M$ a finitely generated $R$-module of dimension $r>0.$ 
 Let $R^\prime=R/\h_\m^0(R)$ and $M^\prime=M/\h_{\m}^0(M).$ Then the following conditions are equivalent:
\begin{enumerate}
\item \label{theorem:finOfSetNoeth} $\Delta_R(M)$ is a finite set;
\item \label{theorem:AuIndim1Noeth} $r=1$ and $R^\prime/\ann_{R^\prime}(M^\prime)$ is analytically unramified.
\end{enumerate}
 \end{theorem}
\begin{proof} 
\eqref{theorem:finOfSetNoeth} $\Rightarrow$ \eqref{theorem:AuIndim1Noeth}: Since $\Delta_R(M)$ is finite, by Lemma \ref{lemma:finite implies d=1}, $r=1.$ 
Thus $M^\prime$ is a Cohen-Macaulay $R^\prime$-module of dimension one. From \eqref{equation:RelHilbCoeff} it follows that $|\Delta_R(M)|=|\Delta_R(M^\prime)|.$ This implies that $\Delta_R(M^\prime)$ is a finite set. 
Since $e_1(I,M^\prime)=e_1(IR^\prime,M^\prime),$ we get that $\Delta_{R^\prime}(M^\prime)$ is a finite set. 
 Let $M^\prime=R^\prime x_1+R^\prime x_2+\cdots+R^\prime x_m,$ where $0 \neq R^\prime x_i \subseteq M^\prime $ is a $R^\prime$-submodule of $M^\prime.$  
 Set $N_i:=R^\prime x_i.$ 
 Since $M^\prime$ is Cohen-Macaulay, $N_i$ is a Cohen-Macaulay $R^\prime$-module of dimension one. Hence for every $\m$-primary ideal $I$ in $R,$ $e_1(IR^\prime,N_i) \geq 0$ and by Proposition \ref{proposition:relatingHilbCoeff},  
 \begin{equation*}
   e_1(IR^\prime,N_i) \leq e_1(IR^\prime,M^\prime)+ c_i, 
 \end{equation*}
 for some nonnegative integer $c_i$ which is independent of $I.$ Thus finiteness of the set $\Delta_{R^\prime}(M^\prime)$ implies that the set 
$\Delta_{R^\prime}(N_i)$ is finite for every $i.$ Hence, by Lemma \ref{lemma:cyclic modules}, $R^\prime/\ann_{R^\prime}(R^\prime x_i)$ is analytically unramified 
for each $i.$ Let $I_i=\ann_{R^\prime}(R^\prime x_i).$ Since $\widehat{R^\prime}/I_i\widehat{R^\prime}$ is reduced for each $i$, 
$\widehat{R^\prime}/\left(\mathop\bigcap\limits_{i=1}^m I_i\widehat{R^\prime}\right) $ is reduced. 
Also, as $\widehat{R^\prime}$ is a flat $R^\prime$-module,
\begin{equation*}
 \widehat{\ann_{R^\prime}(M^\prime)}=\ann_{R^\prime}(M^\prime) \widehat{R^\prime}=\left(\bigcap_{i=1}^m I_i\right) \widehat{R^\prime} = \bigcap_{i=1}^m I_i\widehat{R^\prime}. 
\end{equation*}
Hence 
$\widehat{R^\prime}/\widehat{\ann_{R^\prime} {M^\prime}} \cong \widehat{\left(\frac{R^\prime}{\ann_{R^\prime} M^\prime}\right)}$ is reduced. 
Thus $R^\prime/\ann_{R^\prime}(M^\prime)$ is analytically unramified. \\

\eqref{theorem:AuIndim1Noeth} $\Rightarrow$ \eqref{theorem:finOfSetNoeth}: 
Since $\dim R^\prime/\ann_{R^\prime}(M^\prime)=\dim M^\prime=1$ and $R^\prime/\ann_{R^\prime}(M^\prime)$ is analytically unramified,  
% implies $\ell_{R^\prime}(\overline{R^\prime}/R^\prime)<\infty$.  
by Proposition \ref{proposition:SupAndInfOfSet}, $\Delta_{\frac{R^\prime}{\ann_{R^\prime}(M^\prime)}}(M^\prime)$ is finite. This implies that $\Delta_{R^\prime}(M^\prime)$ is finite. Hence by \eqref{equation:RelHilbCoeff},  $\Delta_R(M)$ is a finite set.
% % % Since $M^\prime$ is a Cohen-Macaulay $R^\prime$-module, by 
% % % Theorem \ref{theorem:finHilCoeffOfModules}, $\Delta_R(M^\prime)=\Delta_{R^\prime}(M^\prime)$ is a finite set. Hence, using \eqref{equation:RelHilbCoeff}, 
% % we get that $\Delta_R(M)$ is a finite set.
\end{proof}

As a consequence we give an equivalent criterion for the finiteness of the set $\Delta^K(R).$

\begin{theorem}\label{theorem:finiteness-of-g1}
 Let $(R,\m)$ be a Noetherian local ring of dimension $d>0$ and $K$ an $\m$-primary ideal of $R.$ Then the following conditions are equivalent:
 \begin{enumerate}
\item \label{theorem:finiteness-of-g1-a} $\Delta^K(R)$ is a finite set;
\item \label{theorem:finiteness-of-g1-b} $d=1$ and $R/\h_\m^0(R)$ is analytically unramified.
\end{enumerate}
\end{theorem}
\begin{proof}
From \eqref{eqn-g1}, $|\Delta^K(R)|=|\Delta_R(K)|.$ Let $R^\prime=R/\h_\m^0(R).$ Since 
$\ann_{R^\prime}(KR^\prime)=0,$ using Theorem \ref{theorem:finHilCoeffOfNoethModules} we get the result.
\end{proof}

\begin{remark}
 Suppose $d=1$ and $R/\h_\m^0(R)$ is analytically unramified. Then by Theorems \ref{KT theorem} and  \ref{theorem:finiteness-of-g1}, the sets $\Delta_R(R)$ and $\Delta^K(R),$ respectively, are 
 finite. Hence, from \eqref{eqn-fi-gi-ei}, the set $\{f_0^K(I)~|~I \mbox{ is an $\m$-primary ideal of } R\}$ is finite.
\end{remark}

In \cite[Corollary 2.4]{KT}, authors gave a description of the set $\Delta_R(R).$ In what follows we will give a description of the set $\Delta^{K}(R).$ 
Recall that a {\it reduction} of an ideal $I$ is an ideal $J\subseteq I$ such that $I^{n+1}=JI^n$ for some $n\geq 0.$ A {\it minimal reduction} of $I$ is a reduction of $I$ which is minimal with respect to inclusion. 
For a minimal reduction $J$ of $I$, {\it reduction number} of $I$ with respect to $J$, denoted by $r_J(I)$, is the least non-negative integer $n$ such that $I^{n+1}=JI^n$.
\begin{theorem}
\label{theorem:DescrofSet}
Let $(R,\m)$ be a Cohen-Macaulay local ring of dimension one with infinite residue field. 
\begin{enumerate}
 \item \label{theorem:DescrofSetforMod}
 For a maximal Cohen-Macaulay module $M$ 
 \begin{equation*}
 \Delta_R(M) \subseteq \{\ell_R(N/M):M \subseteq N \subseteq S^{-1}M, \mbox{ $N$ is a finitely generated $R$-module}\},
\end{equation*}
where $ S=\{x \in R: x\mbox{ is $R$-regular}\}.$
\item \label{theorem:DescrofSetforIdeal} For an $\m$-primary ideal $K$ in $R,$ 
\begin{equation}
\label{equation:DescrofSetforIdeal}
 \Delta^K(R)= \{\ell_R(KB/K)-\ell_R(R/K): R\subseteq B \subseteq \ov{R}, \mbox{ $B$ is a finitely generated $R$-module}\}.
 \end{equation} 
Further, $\sup \Delta^K(R)=\ell_R(K\ov{R}/K)-\ell_R(R/K).$
 \end{enumerate}
\end{theorem}
\begin{proof}
\eqref{theorem:DescrofSetforMod}: Let $I$ be an $\m$-primary ideal in $R.$ Let $J=(x) \subseteq I$ be a minimal reduction of $I.$ Since $R$ (resp. $M$) is 
Cohen-Macaulay, $x$ is $R$-(resp. $M$-)regular. 
 We set 
 \begin{equation*}
  \frac{I^n}{x^n}=\left\{\frac{a}{x^n}:a \in I^n\right\}\subseteq S^{-1}R.
 \end{equation*}
Let $s=r_J(I)$ and $N=M[\frac{I}{x}] \subseteq S^{-1}M.$ Then $M\subseteq N=\mathop\bigcup\limits_{n \geq 0}\frac{I^nM}{x^n}=\frac{I^nM}{x^n} \cong I^nM$ for $n \geq s.$ Thus $N$ is a finitely generated $R$-module. We claim that $e_1(I,M)=\ell_R(N/M).$ We have
\begin{align*}
 \ell_R\left(\frac{M}{I^{n+1}M}\right)&=\ell_R\left(\frac{M}{J^{n+1}M}\right)-\ell_R\left(\frac{I^{n+1}M}{J^{n+1}M}\right)\\
 &=e_0(I,M)(n+1)-\ell_R\left(\frac{I^{n+1}M}{J^{n+1}M}\right) \hspace*{0.8in} \mbox{ for }n \gg 0. 
\end{align*} 
This implies that $e_1(I,M)=\ell_R\left(\frac{I^{n+1}M}{J^{n+1}M}\right)$ for $n \gg 0.$ Since $\frac{I^{n+1}M}{J^{n+1}M} \cong \frac{N}{M} $ for $n \gg 0,$  $e_1(I,M)=\ell_R(N/M).$

\eqref{theorem:DescrofSetforIdeal}: 
Let $\Gamma(R):=\{\ell_R(KB/K): R\subseteq B \subseteq \ov{R}, \mbox{ $B$ is a finitely generated $R$-module}\}.$ First we show that $\Delta_R(K)=\Gamma(R).$ By part \eqref{theorem:DescrofSetforMod}, $e_1(I,K)=\ell_R(N/K)$, where $N=K[\frac{I}{x}]$ and $(x)$ is a minimal 
reduction of $I.$ Put $B=R[\frac{I}{x}]$. Let $s=r_{(x)}(I).$ Then $B=\frac{I^n}{x^n}\cong I^s$ for all $n\geq s$. Thus $B$ is a finitely generated
 $R$-module which implies that $B\subseteq \ov{R}.$ Also, $KB=K[\frac{I}{x}]=N.$ Hence  $e_1(I,K)=\ell_R(KB/K)\in \Gamma(R).$
 
 Now, let $R\subseteq B\subseteq \ov{R}$ and  $B$ is finitely generated
 $R$-module. Then there exists a nonzerodivisor $x\in R$ such $xB\subseteq R$. Let $I=xB$. Then $I$ is an $\m$-primary ideal in $R$ and $I^2=xI.$ Hence 
 $R[\frac{I}{x}]=\frac{I}{x}=B.$ A similar argument as above shows that $e_1(I,K)=\ell_R(KB/K).$ Hence $\Gamma(R) \subseteq \Delta_R(K).$ Thus $\Gamma(R) =\Delta_R(K).$
Therefore using \eqref{eqn-g1}, \eqref{equation:DescrofSetforIdeal} follows. 

Let $C:=\sup \Delta^K(R).$ 
From \eqref{equation:DescrofSetforIdeal} it follows that $C \leq \ell_R(K\ov{R}/K)-\ell_R(R/K).$ 
Hence in order to prove the second assertion we may assume that $C$ is finite. Then, by Theorem \ref{theorem:finiteness-of-g1}, $R$ is analytically unramified and hence $\ov{R}$ is a finite $R$-module. Again using \eqref{equation:DescrofSetforIdeal}, 
we get $C \geq \ell_R(K\ov{R}/K)-\ell_R(R/K).$ 
% Also, $g_1^K(I) \leq \ell_R(K\ov{R}/K)-\ell_R(R/K)$ for every $\m$-primary ideal implies that 
\end{proof} 

\begin{remark}
\begin{enumerate}
 \item 
The containment in Theorem \ref{theorem:DescrofSet}\eqref{theorem:DescrofSetforMod} can be strict. Let $R$ be a Cohen-Macaulay local ring of dimension one and $I$ an $\m$-primary ideal. Choose an integer $t$ such that 
 $e_1(I,R)$ is not divisible by $t.$ Let $M=R^t.$ Then $e_1(J,M)=te_1(J,R)$ for every $\m$-primary ideal $J$ in $R.$ By \cite[Corollary 2.4]{KT}, $e_1(I,R)=\ell_R(B/R)$, for a finite $R$-module $B$ such that $R \subseteq B \subseteq S^{-1}R.$ Now, $R^t \subseteq N:=B \oplus R\oplus\cdots \oplus R \subseteq (S^{-1}R)^t$ and 
$N$ is a finite $R$-module. Also, $\ell_R(N/M)=\ell_R(B/R)=e_1(I,R).$ Suppose there exists an $\m$-primary ideal $J$ in $R$ such that $e_1(J,M)=\ell_R(N/M)=e_1(I,R).$ Then $te_1(J,R)=e_1(I,R)$ which is a contradiction. This implies that the containment in Theorem \ref{theorem:DescrofSet}\eqref{theorem:DescrofSetforMod} can be strict.
\item Let $R$ be a Cohen-Macaulay local ring of dimension one and $M=R^t.$ In this case, $ \Delta_R(M)=\{te_1(I,R): I \mbox{ is an $\m$-primary ideal in $R$}\}.$ Hence, by \cite[Theorem 1.2]{KT}, $\sup \Delta_R(M)=t\ell_R(\overline{R}/R)=\ell_R(\overline{R}/R) \mu_R(M)$, which shows that the bound in Proposition \ref{proposition:SupAndInfOfSet}\eqref{proposition:SupOfSet} can be achieved.
\end{enumerate}
\end{remark}

\section*{Acknowledgements}
The first author is grateful to Prof. Santiago Zarzuela for insightful discussions on section 5. His ideas were helpful to get the results in section 5. She thanks the Institute of Mathematical Sciences (IMSc),
Chennai for supporting her travel to the Centre de Recerca Matem\`atica (CRM), Barcelona during which some part of the work is done. She also thanks CRM for providing local hospitality during the visit. The second author is indebted to her advisor Prof. Anupam Saikia for the encouragement to pursue
this work. She also thanks the Indian Institute of Technology, Guwahati
for granting the Ph.D. scholarship and IMSc for its hospitality where a significant part of this work was discussed. 

\newcommand{\etalchar}[1]{$^{#1}$}

\end{document}